\documentclass[letterpaper, reqno]{amsart}

\pdfoutput=1 

\usepackage{amsmath, amssymb, amsthm}
\usepackage{dsfont}
\usepackage[hidelinks]{hyperref}

\usepackage{graphicx}
\usepackage{cite}
\usepackage{enumerate}
\usepackage{bookmark}
\usepackage{tikz-cd} 
\usepackage{subcaption}
\usepackage{mathtools}
\usepackage{cleveref}
\usepackage{caption}
\usepackage{subcaption}

\usepackage{tikz}
\usepackage{pgfplots}
\pgfplotsset{compat=1.18}
\usepackage{amsmath}
\usetikzlibrary{decorations.pathreplacing}
\usetikzlibrary{patterns}

\crefformat{equation}{(#2#1#3)}
\crefrangeformat{equation}{(#3#1#4)--(#5#2#6)}
\crefmultiformat{equation}{(#2#1#3)}%
{ and~(#2#1#3)}{, (#2#1#3)}{ and~(#2#1#3)}

\crefrangeformat{thm}{Theorems #3#1#4 to #5#2#6}
\crefmultiformat{thm}{Theorems #2#1#3}%
{ and~#2#1#3}{, #2#1#3}{ and~#2#1#3}

\usepackage[activate={true,nocompatibility},final,tracking=true,kerning=true,spacing=true,factor=1100,stretch=20,shrink=20]{microtype}
\microtypecontext{spacing=nonfrench}

\DeclareMathOperator{\supp}{supp}

\DeclareMathOperator{\sgn}{sgn}

\newcommand{\jBra}[1]{\langle #1 \rangle}

\newcommand{\bbR}{\mathbb{R}}

\newcommand{\bbC}{\mathbb{C}}
\newcommand{\bbOne}{\mathds{1}}

\newcommand{\cF}{\mathcal{F}}

\newtheorem{thm}{Theorem}

\newtheorem{lemma}[thm]{Lemma}
\newtheorem{cor}[thm]{Corollary}
\newtheorem{rmk}{Remark}

\begin{document}
\title[Spatial decay for Benjamin-Ono coherent states]{On spatial decay for coherent states of the Benjamin-Ono equation}
\author[Gavin Stewart]{Gavin Stewart}
 \address[Gavin Stewart]{\newline
        Department of Mathematics, \newline
         Rutgers University, New Brunswick, NJ 08903 USA.}
  \email[]{gavin.stewart@rutgers.edu}
\date{July 30th, 2025}
\begin{abstract}
    We consider solutions to the Benjamin-Ono equation
    $$\partial_t u - H \partial_x^2 u = -\partial_x(u^2)$$
    that are localized in a reference frame moving to the right with constant speed.  We show that any such solution that decays at least like $\jBra{x}^{-1-\epsilon}$ for some $\epsilon > 0$ in a comoving coordinate frame must in fact decay like $\jBra{x}^{-2}$.  In view of the explicit soliton solutions, this decay rate is sharp.  Our proof has two main ingredients.  The first is microlocal dispersive estimates for the Benjamin-Ono equation in a moving frame, which allow us to prove spatial decay of the solution provided the nonlinearity has sufficient decay.  The second is a careful normal form analysis, which allows us to obtain rapid decay of the nonlinearity for a transformed equation assuming only modest decay of the solution.  Our arguments are entirely time dependent, and do not require the solution to be an exact traveling wave.
\end{abstract}
\maketitle

\section{Introduction}


\subsection{Background}
We will study coherent states for the Benjamin-Ono equation
\begin{equation}\label{eqn:BO}
    \partial_t u - H \partial_x^2 u = -\partial_x(u^2)
\end{equation}
where $H$ is the Hilbert transform
\begin{equation*}
    \cF(Hf)(\xi) = -i\sgn(\xi) \hat{f}(\xi)
\end{equation*}
The Benjamin-Ono equation was proposed by Benjamin~\cite{benjaminInternalWavesPermanent1967a} and Davis-Acrivos~\cite{davisSolitaryInternalWaves1967a} as a reduced order model for internal waves in two-fluid systems (see the recent work~\cite{paulsenJustificationBenjaminOno2024a} for a mathematically rigorous derivation in this context).  It was also recently shown to describe waves in two-dimensional gravity water waves with constant vorticity up to the natural cubic timescale~\cite{ifrimBenjaminOnoApproximation2D2022b}.  

The Benjamin-Ono equation admits traveling wave solutions
\begin{equation}
    S_c(x,t) = \frac{2c}{c^2(x-ct)^2 + 1}
\end{equation}
which travel to the right with speed $c$.  In contrast to other dispersive models, the solitons for the Benjamin-Ono equation decay only algebraically at spatial infinity, which can be seen as a consequence of the nonsmooth dispersion relation
\begin{equation}\label{eqn:BO-disp-rel}
    \omega(\xi) = |\xi|\xi
\end{equation}
appearing in the symbol $H\partial_x^2 = \omega(D_x)$, $D_x = \frac{1}{i}\partial_x$.  The existence of solitons solutions can be seen as a consequence of the complete integrability of the Benjamin-Ono equations~\cite{fokasBiHamiltonianFormulationKadomtsev1988,fokasInverseScatteringTransform1983a,kaupCompleteIntegrabilityBenjaminOno1998a}, as predicted in~\cite{onoAlgebraicSolitaryWaves1975a}.

Based on the soliton resolution conjecture, we expect that a generic solution to~\eqref{eqn:BO} that is localized in a moving reference must be equal to $S_c$ (up to translation).  For the case of small perturbations of a soliton in $H^{1/2}$, this result was shown by Kenig and Martel in~\cite{kenigAsymptoticStabilitySolitons2009}.  Proving the result in full generality appears difficult; although Benjamin-Ono is completely integrable, using the complete integrability to obtain results about the dynamics is technically quite involved.  Our goal in this paper will be to establish the following result (for a precise statement, see~\Cref{thm:main,thm:main-symb-bds}):
\begin{thm}[Main theorem, rough version]\label{thm:rough-main}
    If $u$ is a solution to~\eqref{eqn:BO} that is localized in a reference frame moving to the right at constant speed, then the tails of $u$ decay at the same rate as the tails of $S_c$ as $|x| \to \infty$.
\end{thm}
In particular, our theorem does not require that $u$ be close to a soliton.  

The idea of finding spatial bounds for localized solutions to dispersive equations goes back to Tao's work in~\cite{taoGlobalCompactAttractor2008a} for nonlinear Schr\"odinger equations, and was recently revisited by Soffer and Wu in~\cite{soffer2023soliton}.  These works are only able to prove polynomial decay at infinity, which falls far short of the exponential decay of the NLS solitons.  The main obstruction to obtaining better estimates for Schr\"odinger-type equations are low frequencies, which have small group velocities.  For high frequencies (large group velocities), work by Soffer and the author shows that it is possible to get faster decay using an incoming/outgoing decomposition~\cite{sofferScatteringLocalizedStates2024a}.

For the Benjamin-Ono equation, we get improvements because the dispersion relation~\eqref{eqn:BO-disp-rel} only supports (linear) waves traveling to the left.  Thus, if we pass to a frame moving with speed $c$ to the right, all linear waves will move to the left with speed $\geq c$, leading to improved dispersive decay in weighted spaces.  This idea goes back at least to the work of Pego and Weinstein on the asymptotic stability of solitary waves for KdV-type equation~\cite{pegoAsymptoticStabilitySolitary1994}.  Here, we use it to obtain the estimate
\begin{equation}\label{eqn:local-decay-intro}\begin{split}
    \left\lVert \jBra{x}^a \int_{-\infty}^t e^{(t-s)(c\partial_x + H \partial_x^2)} F(s)\;ds \right\rVert_{L^\infty} \lesssim \lVert \jBra{x}^{a+3/2} F(t) \rVert_{L^\infty_{x,t}}
\end{split}\end{equation}
for $a \leq 2$.  (In fact, we can also obtain a $1/2$ derivative smoothing effect if we localize appropriately in space and frequency; see~\Cref{sec:lin-ests}.)  Here $e^{t(c \partial_x + H \partial_x^2)}$ is the linear propagator for the Benjamin-Ono equation in a reference frame moving to the right at constant speed $c$.  If $w(x,t) = u(x+ct, t)$ is the solution to~\eqref{eqn:BO} in the moving reference frame, then the polynomial nonlinearity is given by
\begin{equation*}
    F = -\partial_x(w^2)
\end{equation*}
Because of the derivative, the nonlinearity is (mildly) quasilinear, so in our analysis we will need some control of higher derivatives of $w$ in order to close.  However, even at medium frequencies, the quadratic nonlinearity causes problems: under the assumption $|w| \lesssim \jBra{x}^{-1-\epsilon}$, the projection of $F$ to medium frequencies still only decays like $|w|^2 \lesssim \jBra{x}^{-2-2\epsilon}$, so by~\eqref{eqn:local-decay-intro}, the best bound we can expect is
$$|w| \lesssim \jBra{x}^{-1/2-2\epsilon}$$
which does not improve on our bootstrap assumption $|w| \lesssim \jBra{x}^{-1-\epsilon}$ unless $\epsilon > \frac{1}{2}$.  In particular, linear estimates alone are not enough to prove~\Cref{thm:rough-main} in the full scaling subcritical range $\epsilon > 0$, even ignoring the presence of derivatives in the nonlinearity.  

The main obstacle we face is that the quadratic nonlinearity does not produce enough decay to bootstrap if the decay of the solution $w$ is sufficiently slow.  To address this problem, we take advantage of the fact that~\eqref{eqn:BO} admits a normal form transformation, which (roughly speaking) allows us to replace the quadratic nonlinearity in~\eqref{eqn:BO} with cubic and higher order terms.  In some sense, the utility of this normal form transformation should not be surprising: Benjamin-Ono is a model for water waves, and beginning with the work of Wu~\cite{wuAlmostGlobalWellposedness2009,wuGlobalWellposedness3D2011} and Germain, Shatah, and Masmoudi~\cite{germainGlobalSolutionsGravity2012}, normal forms have become a standard tool to obtain extra decay for water wave equations~\cite{ifrimTwoDimensionalWater2016,ifrimLifespanSmallData2017a,ifrimTwoDimensionalGravity2019,hunterTwoDimensionalWater2016,harrop-griffithsFiniteDepthGravity2017,alazardGlobalSolutionsAsymptotic2015,wangGlobalSolution3D2019}.  Note, however, that in contrast to previous work, which used normal forms to prove improved dispersive decay in time for small and localized solutions, we use normal forms to obtain improved \textit{spatial} decay of the nonlinearity.

The normal form transformation for Benjamin-Ono is degenerate due to bad high-low frequency interactions in the nonlinearity.  To overcome this difficulty, Ifrim and Tataru in~\cite{mihaelaifrimWellposednessDispersiveDecay2019} introduced a paradifferential normal form/gauge transformation of the form\footnote{Note that our sign conventions for the Benjamin-Ono equation~\eqref{eqn:BO} differ from those of~\cite{mihaelaifrimWellposednessDispersiveDecay2019}.  Here, to avoid confusion, we present the version of their normal form/gauge transformation adapted to our sign convention.}
$$v_k^+ = (u_k^+ + B_k^+(u,u)) e^{-i \Phi_{\ll k}}$$
where $u_k^+ = P_k^+ u$ is the projection of $u$ to positive frequencies $\xi \sim 2^k$ and $\Phi = \partial_x^{-1} u$.  The transformed variable $v_k^+$ then satisfies a nonlinear Schr\"odinger equation
\begin{equation}\label{eqn:v-k-p-NFT-intro}
    (i\partial_t - \partial_x^2) v_k^+ = C_k(u,u,u) e^{-i \Phi_{\ll k}} + Q_k(u,u,u,u)e^{-i \Phi_{\ll k}}
\end{equation}
where $C_k$ and $Q_k$ are cubic and quartic pseudoproducts.  This might appear to be an improvement, since the spatial decay for $u$ now translates into a stronger spatial decay for the nonlinear terms.  However, the equation for $v_k^+$ does not have unidirectional wave propagation: the linear part (which is a Schr\"odinger equation) supports waves traveling in both directions, and the exponential gauge transformation is not frequency localized, meaning that we cannot get estimates like~\eqref{eqn:local-decay-intro} for~\eqref{eqn:v-k-p-NFT-intro}.  To get around this obstacle, we approximate the exponential gauge transform by a high-order Taylor polynomial:
\begin{equation*}
    e^{-i \Phi_{\ll k}} \approx \sum_{n=0}^N \frac{\left(-i \Phi_{\ll k}\right)^n}{n!} =: E_n(\Phi_{\ll k})
\end{equation*}
The approximate gauge transformation will be frequency localized provided that the frequency localization of $\Phi_{\ll k}$ is chosen to depend on $N$.  Substituting $E_n(\Phi_{\ll k})$ for $e^{-i\Phi_{\ll k}}$ gives a modification of~\eqref{eqn:v-k-p-NFT-intro} that can then be controlled by arguments similar to~\eqref{eqn:local-decay-intro}.  In particular, if we fix a spatial scale $2^j$, then for frequencies above a certain threshold $k_0$, we can show that 
$$\lVert \jBra{x-ct}^{1+\epsilon} u \rVert_{L^\infty} < \infty \implies \sup_{|x-ct|
\sim 2^j} \sum_{k > k_0} |v_k^+| \lesssim 2^{-(1+3/2\epsilon)j}$$
where $k_0 = -\frac{1-\epsilon}{2}j$ depends on $j$ and $\epsilon$.  Moreover, $B_k(u,u)$ is a pseudoproduct of order $-1$, so for $|x-ct| \sim 2^j$,
$$|u_k^+| \sim |v_k^+| + O(2^{-k} 2^{-2(1+\epsilon)j}) \lesssim 2^{-(1+3/2\epsilon)j}$$
which gives us control decay for non-small frequencies.  On the other hand, for low frequencies $|\xi| < 2^{k_0}$, the derivative in~\eqref{eqn:BO} provides additional decay, which allows us to bootstrap an estimate of the form
$$\lVert \jBra{x-ct}^{1+\epsilon} u(x,t) \rVert_{L^\infty_{x,t}} < \infty \implies  \lVert \jBra{x-ct}^{\min(2, 1+3/2\epsilon)} u(x,t) \rVert_{L^\infty_{x,t}} < \infty$$
In particular, this argument can be iterated (finitely many times) to prove that $u$ decays at least like $\jBra{x-ct}^{-2}$.

\subsection{Main results}
The main result of this paper is that localized solutions which decay a least like $\jBra{x-t}^{-1-\epsilon}$ must in fact decay like $\jBra{x-t}^{-2}$:
\begin{thm}\label{thm:main}
    Suppose $u \in L^\infty_t B^\frac{1}{2}_{2,1}$ is a solution to~\eqref{eqn:BO} such that for all $t \in \bbR$,
    \begin{equation}\label{eqn:main-decay-hypo}
        |u(x,t)| \lesssim \jBra{x-t}^{-1-\epsilon}
    \end{equation}
    for some $\epsilon > 0$.  Then, 
    \begin{equation}\label{eqn:main-decay-undiff}
        |u(x,t)| \lesssim \jBra{x-t}^{-2}
    \end{equation}
    for all times $t$.
\end{thm}
Although this theorem is specialized to the case of a reference frame moving at speed $1$, by applying the scaling transformation
$$u \mapsto \lambda u(\lambda x, \lambda^2 t)$$
we see that~\Cref{thm:main} also implies that a solution that is localized in a reference frame moving with any speed speed $c = \lambda^{-1} > 0$ to the right decays like $\jBra{x-ct}^{-2}$.  In the endpoint case $c = 0$, the results of~\cite{munozAsymptoticBehaviorSolutions2019} imply that the only solution that is localized in a stationary reference frame is $u(x,t) \equiv 0$.

The condition $u \in L^\infty_t B^\frac{1}{2}_{2,1}$ can be seen as a Besov refinement of the condition that $u$ is finite energy, since the energy space for~\eqref{eqn:BO} is $H^\frac{1}{2}$.  In particular, we observe that the Benjamin-Ono soliton is in $B^{1/2}_{2,1}$ (in fact, it is in $H^\infty$).  


If we assume more differentiability of $u$, then we can prove that derivatives of $u$ obey symbol-type decay bounds in a moving reference frame:
\begin{thm}\label{thm:main-symb-bds}
    Let $n \geq 1$ be a positive integer.  Then, if $u \in L^\infty_t B^{n+\frac{1}{2}}_{2,1}$ is a solution to~\eqref{eqn:BO}
    satisfying the bounds
    \begin{equation}\label{eqn:symb-decay-hypo}
        \sup_{m \leq n} |\partial_x^m u(x,t)| \lesssim \jBra{x-t}^{-1-\epsilon}
    \end{equation}
    uniformly in time, then for $m =1 , 2, \cdots, n$
    \begin{equation}\label{eqn:symb-decay-bd}
        |\partial_x^m u(x,t)| \lesssim \jBra{x-t}^{-m-2}
    \end{equation}
\end{thm}

\subsection{Sketch of the proof}

Let $w(x,t) = u(x+t, t)$ denote the solution in a rightward-moving reference frame:
\begin{equation}\label{eqn:BO-moving-frame-intro}
    \partial_t w -c \partial_x w - H \partial_x^2 w = -\partial_x (w^2)
\end{equation}
Duhamel's formula shows that
\begin{equation*}
    w(x,t) = -\int_{-\infty}^t e^{(t-s)(\partial_x + H\partial_x^2)} \partial_x w(x,s)^2\;ds
\end{equation*}

\begin{figure} 
\begin{tikzpicture}
  \def\myfunction{1/(1+x^2)}
  \begin{axis}[
      axis x line=center,
      axis y line=none,
      axis line style={-},
      xtick=\empty,
      ytick=\empty,
      ylabel={},
      xlabel={},
      ymin=-0.4,
      ymax=1.2,
      xmin=-4,
      xmax=10,
      samples=200,
      domain=-4:10,
      enlarge x limits=false,
      axis equal image=false
  ]
  \addplot[
      color=black,
      thick 
  ] {\myfunction};

  \draw[dashed] (axis cs:1,0) -- (axis cs:1,{1/(1+1^2)});
  \draw[dashed] (axis cs:2,0) -- (axis cs:2,{1/(1+2^2)});

  \addplot[
      pattern=north west lines,
      pattern color=gray,
      draw=none,
      domain=1:2
  ] {\myfunction} \closedcycle;
  
  \draw [decorate,decoration={brace,amplitude=5pt,mirror,raise=0.5ex}]
    (axis cs:1,-0.05) -- (axis cs:2,-0.05)
    node [midway,below=0.1cm,font=\footnotesize] {$x \sim 2^j$}; 

  \addplot[
      fill=gray,
      fill opacity=0.5,
      draw=none,
      domain=4:8
  ] {\myfunction} \closedcycle;

  \addplot[
      pattern=crosshatch dots,
      pattern color=gray,
      draw=none,
      domain=-4:0.5
  ] {\myfunction} \closedcycle;

  \draw[->,thick] (axis cs: -2, 0.2) -- (axis cs: -3, 0.2);

  \draw[->,thick] (axis cs: 4, 0.0303) -- (axis cs: 3, 0.0303);

  \draw [decorate,decoration={brace,amplitude=5pt,mirror,raise=0.5ex}]
    (axis cs:4,-0.05) -- (axis cs:8,-0.05)
    node [midway,below=0.1cm,font=\footnotesize] {$x \sim 2^\ell$}; 

  \draw [decorate,decoration={brace,amplitude=5pt,mirror,raise=0.5ex}]
    (axis cs:-4,-0.05) -- (axis cs:0.5,-0.05)
    node [midway,below=0.1cm,font=\footnotesize] {$x < 2^{j-10}$}; 

  \end{axis}
\end{tikzpicture}

\caption{\label{fig:disp} A depiction of a solution to~\eqref{eqn:BO-moving-frame-intro}.  If we are interested in the size of the solution in the region $\mathcal{R}_j^+ = \{x \sim 2^j, x > 0\}$ (cross-hatch shading), then because of the moving reference frame and unidirectional wave propagation, any wave initially localized to the left of $x = 2^{j-10}$ will not contribute (dotted region), while any waves initially localized in $x \sim 2^\ell, x > 0$ will have exited $\mathcal{R}_j^+$ for times $t \gg 2^\ell$.}
\end{figure}
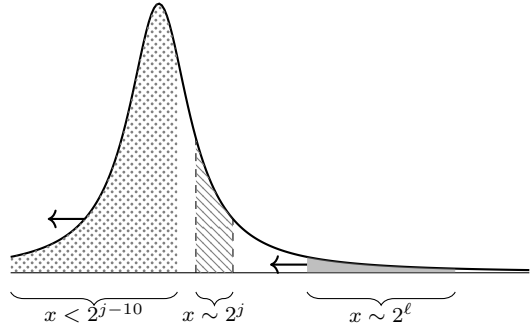

For the purpose of exposition, let us ignore all issues of regularity by projecting to frequencies $\xi = O(1)$, which gives the equation
\begin{equation*}
    W(x,t) = -\int_{-\infty}^t e^{(t-s)(\partial_x + H\partial_x^2)} F(x,s)\;ds
\end{equation*}
for $W = P_0 w$, with
\begin{equation*}
    F(x,t) = -\partial_x P_0 (w(x,t)^2)
\end{equation*}  
If we assume that $w$ decays like $\jBra{x}^{-1-\epsilon}$ for $\epsilon > 0$, then $\jBra{x}^\alpha F \in L^\infty_{x,s}$ for $\alpha = -2 - 2\epsilon$.  Since the group velocity for the propagator $e^{t(\partial_x + H \partial_x^2)}$ is greater than or equal to $1$ at all frequencies, we do not expect waves initially localized in $x < 0$ or $|x| \ll 2^j$ to contribute significantly to the Duhamel integral for $x > 0, x \sim 2^j$.  Thus, the main contribution here will come from waves initially localized at $x \sim 2^\ell, x> 0$ for $\ell > j - 10$.  Since these waves travel with speed at least $1$, they will have completely passed through the region $x \sim 2^j$ after time $t = O(2^\ell)$.  (See~\Cref{fig:disp}.)  Summing all these contributions gives
\begin{equation*}
    \lVert \chi_j^+ W(x,t) \rVert_{L^\infty} \leq \sum_{\ell > j - 10} \int_{t-O(2^{\ell})}^{t} \lVert e^{(t-s)(\partial_x + H \partial_x^2)} \chi_\ell^+(x) F(x,s)\rVert_{L^\infty}\;ds + \{\textup{better}\}
\end{equation*}
where $\chi_j^+$ is the projection to the dyadic region $x \sim 2^j, x > 0$.  Using the $L^1 \to L^\infty$ dispersive decay bounds and assuming $\alpha > 3/2$, we obtain the local decay estimate
\begin{equation}\label{eqn:local-decay-intro-2}\begin{split}
    \lVert \chi_j^+ W(x,t) \rVert_{L^\infty} \leq& \sum_{\ell > j - 10} \int_{t-O(2^{\ell})}^{t} \lVert e^{(t-s)(\partial_x + H \partial_x^2)} \chi_\ell^+(x) F(x,s)\rVert_{L^\infty}\;ds + \{\textup{better}\}\\
    \lesssim& \sum_{\ell > j - 10} \int_{t-O(2^{\ell})}^{t} (t-s)^{-1/2} 2^{(\alpha - 1)\ell}\lVert \jBra{x}^{\alpha} F(x,s)\rVert_{L^\infty}\;ds + \{\textup{better}\}\\
    \lesssim& (2^{-2j} + 2^{(\alpha-3/2)j})\lVert \jBra{x}^a F \rVert_{L^\infty_{x,s}}
\end{split}\end{equation}
where the $2^{-2j}$ factor comes from the $\{\textup{better}\}$ terms due to the nonsmooth dispersion relation (see~\Cref{thm:ref-lin-est} for a precise derivation).  We remark that the idea that unidirectional wave motion leads to improved dispersive estimates between weighted spaces goes back at least to the work of Pego and Weinstein~\cite{pegoAsymptoticStabilitySolitary1994}.  Recalling that $\alpha = -2 - 2\epsilon$, we see that the Duhamel integral improves on our original decay assumptions when
\begin{equation*}
    \alpha + 3/2 < -1 - \epsilon \Leftrightarrow \epsilon > 1/2
\end{equation*}
Assuming for a moment that we could get similar improvements at high frequencies, we could then repeat the argument to bootstrap to inverse square decay.  In particular, this can be seen as a time-dependent, dispersive analog of the elliptic results on the decay of solitary waves obtained in~\cite{hurNoSolitaryWaves2012,sunAnalyticalPropertiesCapillarygravity1997a}.  

The critical scaling for~\eqref{eqn:BO} in the scale of spaces $|x|^{-1-\epsilon}L^\infty_{x,t}$ is $\epsilon = 0$, so this (medium frequency) analysis only works when we are more than half a space weight above scaling.  The main obstacle to pushing $\epsilon$ lower is that the quadratic nonlinearity does not produce enough decay to compensate for the $3/2$ loss of weights in~\eqref{eqn:local-decay-intro-2} when $\epsilon$ is small.  To get around this problem, we will use a normal form transformation to obtain a higher order nonlinearity, which allows us to bootstrap from lower decay.  A natural candidate is the paradifferential normal form/gauge transformation for~\eqref{eqn:BO} introduced by Ifrim and Tataru in~\cite{mihaelaifrimWellposednessDispersiveDecay2019} (see also~\cite{taoGlobalWellposednessBenjamin2004,ionescuGlobalWellposednessBenjamin2007,molinetCauchyProblemBenjamin2012a}), which takes the form
\begin{equation*}
    v_k^+ = (u_k^+ + B_k^+(u,u)) e^{-i\Phi_{\ll k}}
\end{equation*}
where $u_k^+ = P_k^+ u$ is the projection of $u$ to positive frequencies of size $\sim 2^k$, and $\Phi = \partial_x^{-1} u$.  Under this change of variables, $v_k^+$ satisfies a Schr\"odinger equation
\begin{equation}\label{eqn:IT-NFT}
    i\partial_t v_k^+ - \partial_x^2 v_k^+ = (C_k(u,u,u) + Q_k(u,u,u,u)) e^{-i\Phi_{\ll k}} 
\end{equation}
for some cubic and quartic order $0$ pseudoproduct operators $C_k$ and $Q_k$.  (Note that our sign conventions in~\eqref{eqn:BO} differ from~\cite{mihaelaifrimWellposednessDispersiveDecay2019}, so their $v_k^+$ is related to ours by time-reversal.)

At first glance, this appears to solve our problems: Since $C_k$ and $Q_k$ are pseudolocal (see~\Cref{sec:pseudo-prod-ops}), the nonlinear terms in~\eqref{eqn:IT-NFT} satisfy better bounds than the quadratic term in~\eqref{eqn:BO} (and in addition have better regularity).  Based on this, we might expect $v_k^+$ to satisfy better weighted bounds than $u_k^+$.

The problem with this assessment is that the linear part of~\eqref{eqn:IT-NFT} is no longer unidirectional, so local decay bounds of the form~\eqref{eqn:local-decay-intro-2} can only hold under localization to positive frequencies.  The pseudoproducts $C_k$ and $Q_k$ will be localized to frequencies $\xi \sim 2^k, \xi > 0$; however, this localization is lost when we multiplying by $\exp(-i\Phi_{\ll k})$, which has no localization in frequency because of the exponential.

To overcome this difficulty, we replace the exponential by its high-order Taylor expansion and define
\begin{equation}\label{eqn:NFT-approx-gauge-intro}
    v_k^+ = (u_k^+ + B_k^+(u,u)) E_N(\Phi_{\ll k})
\end{equation}
with $E_N(x) = \sum_{n=0}^N \frac{(-ix)^n}{n!}$.  The transformed variable $v_k^+$ satisfies
\begin{equation}\label{eqn:AGTI-NFT}\begin{split}
    i\partial_t v_k^+ - \partial_x^2 =& \mathcal{B}_k^+(u,u) (-i \Phi_{\ll k})^N + C_k^+(u,u,u) E_{N-1}(\Phi_{\ll k}) + \{\text{better}\}
\end{split}\end{equation}
As long as $\Phi_{\ll k}$ is localized to sufficiently low frequencies (depending on $N$), all the terms in this expansion will be localized to positive frequencies $\xi \sim 2^k$, allowing us to treat the Schr\"odinger equation as having unidirectional group velocity.  The price to pay is that we cannot completely eliminate the quadratic-in-$u$ terms, since we only perform an approximate gauge transformation.  However, these terms turn out to be harmless from the perspective of decay: Since $\Phi = \int^\infty_x u(y)\;dy$ decays like $\jBra{x-t}^{-\epsilon}$ under the bootstrap hypotheses, the quadratic remainder can be made to decay arbitrarily fast in $x$.  However, they contain a derivative in an unfavorable position, so here (and only here) we must use hypothesis $u \in L^\infty_t B^{1/2}_{2,1}$ (combined with a $1/2$ derivative smoothing estimate for the linear Schr\"odinger equation) to close the estimates.

\begin{rmk}
    The idea of approximating the gauge transformation by a high-order power series works here because the gauge transformation can be seen as the limit of an infinite sequence of normal form transformations of the form~\eqref{eqn:NFT-approx-gauge-intro}.  See~\cite{correiaNonlinearSmoothingUnconditional2021} for an application of this idea to nonlinear smoothing for the Benjamin-Ono equation.  
\end{rmk}

With the gauge transformation~\eqref{eqn:NFT-approx-gauge-intro} in hand, our basic argument is as follows: First, we must show that $\chi_j^+(x)|B_k^+(w,w)|$ decays rapidly in $x$, since in that case~\eqref{eqn:NFT-approx-gauge-intro} implies that
$$\chi_j^+(x) w_k^+ = \chi_j^+(x)\tilde{w}_k^+ + \{\textup{lower order}\}$$
where $\tilde{w}_k^+ = v_k^+(x+t, t)$ is $v_k^+$ in the coordinates of the moving reference frame.  Second, we must show that the nonlinear terms in ~\eqref{eqn:AGTI-NFT} decay rapidly enough that the local decay estimate~\eqref{eqn:local-decay-intro-2} (or, more precisely, its version for frequency localized Schr\"odinger equations) gives decay for $\tilde{w}_k^+$ which is faster than for $w_k^+$.  The chief obstacle we face here is that the nonlinear terms $B_k^+$, $C_k^+$, and the like contain Hilbert transforms, which lead to a loss of decay.  Indeed, using the physical-space representation of the Hilbert transform, we see that
\begin{equation*}
    \lVert [F(x),H] g \rVert_{L^\infty} = \left\lVert \int \frac{F(x) - F(y)}{x-y} g(y)\;dy\right\rVert_{L^\infty} \leq \left\lVert \frac{F(x) - F(y)}{x-y} \right\rVert_{L^\infty_{x,y}} \lVert g \rVert_{L^1}
\end{equation*}
Applying this estimate to $w$ with $F(x) = \chi_j^\pm(x)$ (where $\chi_j^\pm(x)$ a bump function supported on $\pm x \sim 2^j$) and writing
$$\chi_j^\pm(x) H w = [\chi_j^\pm(x), H] w + H \chi_j^\pm(x) w$$
we see that the component $[\chi_j^\pm(x), H] w$ only decays like $O(2^{-j})$ for $x \sim \pm 2^j$, which is significantly slower than the $O(2^{-(1-\epsilon)j})$ estimate for $\chi_j^\pm(x) w$ in this same region.  The term $H \chi_j^\pm(x) w$ provides no useful cancellation in our setting: The Hilbert transformation is bounded as an operator from $L^\infty$ to $\jBra{D}^{0+} \jBra{x}^{0+} L^\infty$, so this term will be lower order once we take into account local smoothing.  Indeed, for the soliton solution $S(x) = 2(x^2 + 1)^{-1}$, classical complex analysis shows that $|HS(x)| \sim |x|^{-1}$ as $|x| \to \infty$, so this slow decay is a genuine feature of the problem.  Thus, when proving decay for the nonlinearity, we must pay careful attention to the exact structure of the nonlinear terms.  After doing so, we find that
$$|B_k^+(w,w)| \lesssim 2^{-(2+\epsilon)j} 2^{(-2+) k}$$
and
$$|C_k^+(w,w,w)| \lesssim 2^{-(3+\epsilon)j}\left(2^{-k} + 2^{(0+)k}\right)$$
which shows allows us to bootstrap better decay for $\chi_j^+(x) w_k^+$ under the assumption $k > -\frac{1-\epsilon}{2}j$.  In the low-frequency regime $k \leq -\frac{1-\epsilon}{2}j$, the derivative in the untransformed Benjamin-Ono equation~\eqref{eqn:BO-moving-frame-intro} gives extra decay, allowing us to close using~\eqref{eqn:local-decay-intro-2}.

\subsection{Plan of the paper}

The plan for the rest of the paper is as follows:
\begin{itemize}
    \item In~\Cref{sec:prelim}, we begin by giving key definitions in~\Cref{sec:prelim-def}.  We then define and prove useful results on pseudoproduct operators in~\Cref{sec:pseudo-prod-ops}.  Finally, we prove several commutator estimates involving the Hilbert transform in~\Cref{sec:H-bdds}.
    \item In~\Cref{sec:lin-ests}, we prove space and frequency localized dispersive estimates of the form~\eqref{eqn:local-decay-intro-2}. 
    \item In~\Cref{sec:NFGT}, we define the paralinear normal form transformation and prove bounds for pseudoproduct terms arising in the transformation and the transformed equation.
    \item In~\Cref{sec:proof-sec}, we first give the proof for~\Cref{thm:main} in~\Cref{sec:undiff-decay}, and then explain in~\Cref{sec:symb-bds} how to modify the proof for the symbols bounds in~\Cref{thm:main-symb-bds}.
\end{itemize}

\section{Preliminaries}\label{sec:prelim}

\subsection{Definitions}
\label{sec:prelim-def}

The Fourier transform of a function $f$ is given by
\begin{equation*}
    \mathcal{F} f(\xi) = \hat{f}(\xi) := \frac{1}{\sqrt{2\pi}} \int e^{-i\xi x} f(x)\;dx
\end{equation*}
with inverse
\begin{equation*}
    \mathcal{F}^{-1} g(x) = \check{g}(x) := \frac{1}{\sqrt{2\pi}} \int e^{i\xi x} g(\xi)\;d\xi
\end{equation*}
Given a function $m : \bbR \to \bbC$, we define the Fourier multiplication operator $m(D)$ by
\begin{equation*}
    m(D) f(x) = \mathcal{F}^{-1} m(\xi) \hat{f}(\xi)
\end{equation*}

Let $\chi^+_{\leq 0}: \bbR \to [0,1]$ be a smooth function with $\chi^+_{\leq 0}(x) = 1$ for $x \leq 1$ and $\chi^+_{\leq 0}(x) = 0$ for $x \geq 2$.  We define
\begin{equation*}\begin{split}
    \chi^+_{\leq j}(x) =& \chi^+_{\leq 0}(2^{-j} x)\\
    \chi^+_{\geq j}(x) =& 1 - \chi^+_{\leq j-1}(x)\\
    \chi^+_j(x) =& \chi^+_{\leq j}(x) - \chi^+_{\leq j-1}(x)\\
\end{split}\end{equation*}
we also define $\chi^-_j(x) = \chi^+_j(-x)$ and $\chi_j(x) = \chi^+_j(|x|)$, and similarly for $\chi_{\leq j}, \chi_{\geq j}$.  We also define the Littlewood-Paley projections as
\begin{equation*}\begin{split}
    P_k =& \chi_k(D)\\
    P_{\leq k} =& \chi_{\leq k}(D)\\
    P_{\geq k} =& \chi_{\geq k}(D)\\
    P^\pm_k =& \bbOne_{\pm D > 0} P_k\\
    P^\pm_{\leq k} =& \bbOne_{\pm D > 0} P_{\leq k}\\
    P^\pm_{\geq k} =& \bbOne_{\pm D > 0} P_{\geq k}
\end{split}\end{equation*}
Note that $P^\pm_{\leq k} \neq \chi^\pm_{\leq j}(D)$ (the former vanishes at positive frequencies, while the latter is equal to $1$ there).  We also adopt the notation
\begin{equation*}
    P^\pm = \bbOne_{\pm D > 0}
\end{equation*}
and note that $(H \pm i) = \pm 2iP^\mp$.

For ease of notation, we will often write $u_k$ to denote $P_k u$, $u_{\leq k} = P_{\leq k} u$, and so on.  To simplify later calculation, we will also adopt the shorthand notation
\begin{equation*}\begin{split}
    \chi^+_{[N,M]}(x) =& \sum_{j=N}^M \chi^+_j(x)\\
    \chi^+_{[N\pm A]}(x) =& \chi^+_{[N-A,N+A]}(x)\\
    \chi^+_{\sim j}(x) =& \chi^+_{[j-10,j+10]}(x)\\
    \chi^+_{\lesssim j}(x) =& \chi^+_{\leq j+10}(x)\\
    \chi^+_{\gtrsim j}(x) =& \chi^+_{\geq j+10}(x)
\end{split}\end{equation*}
and similarly for $\chi^+_{\sim j}(x)$, $P_{\sim k}$, and so on.  For later use, we also define
\begin{equation*}\begin{split}
    \chi_{\ll_N k}(x) =& \chi_{< k - 100N}(x)\\
    \chi_{\gtrsim_N k}(x) =& 1-\chi_{\ll_N k}(x)
\end{split}\end{equation*}

Due to the uncertainty principle, spatial projections (like $\chi^+_j(x)$) and frequency projections (like $P_{\leq k}$) do not commute.  However, we do have the following \textit{pseudolocality} principle: for any $N \geq 1$,
\begin{equation}\label{eqn:LP-pseudolocality}
    \lVert \chi_j^+(x) P_{\leq k} (1 - \chi^+_{\sim j}(x)) \rVert_{L^\infty \to L^\infty} \lesssim_N \jBra{2^{j+k}}^{-N}
\end{equation}
This can be seen as a special case of the more general pseudolocality obeyed by general pseudodifferential operators: see~\cite{zworskiSemiclassicalAnalysis2012}.

As a bookkeeping device when dealing with the Besov regularity in~\Cref{thm:main}, we will often use $c_k$ to denote a term in an $\ell^1_k$ sequence: for instance, we have that
\begin{equation*}
    \lVert \partial_x P_k u \rVert_{L^2} = 2^{k/2} c_k \lVert u \rVert_{B^{1/2}_{2,1}}
\end{equation*}



\subsection{Pseudoproduct operators}\label{sec:pseudo-prod-ops}

Given a symbol $b(\xi, \eta)$, we define the bilinear pseudoproduct operator $B(f,g)$ to be
\begin{equation*}
    B(f,g) = \mathcal{F}^{-1} \int b(\xi, \eta) \hat{f}(\xi-\eta) \hat{g}(\eta)\;d\eta
\end{equation*}
We note that in the case $b = 1$, we have that $B(f,g) = \sqrt{2\pi}fg$ is simply the usual product of functions.  Similarly, we have that
$$P_{k}(f P_{\ll k} g) = B(f,g)$$
is a pseudoproduct operator with $b(\xi,\eta) = \chi_k(\xi) \chi_{\ll k}(\eta)$.  We observe that pseudoproduct operators obey a Leibnitz-type rule:
\begin{equation*}\begin{split}
    \partial_x B(f,g)   =& \mathcal{F}^{-1} \int i\xi b(\xi, \eta) \hat{f}(\xi-\eta) \hat{g}(\eta)\;d\eta\\
                        =& B(\partial_x f, g) + B(f, \partial_x g)
\end{split}\end{equation*}
For reasonably smooth symbols $b$, the pseudoproduct $B(f,g)$ obeys H\"older-type estimates:
\begin{lemma}\label{lem:pseudo-prod-Holder}
    If $\mathcal{F}^{-1} b \in L^1$, then
    $$\lVert B(f,g) \rVert_{L^p} \lesssim \lVert f \rVert_{L^q} \lVert g \rVert_{L^r}$$
    whenever $\frac{1}{p} = \frac{1}{q} + \frac{1}{r}$.
\end{lemma}
\begin{proof}
    Taking the inverse Fourier transform, we see that
    \begin{equation*}
        B(f,g) = 2\pi \iint \mathcal{F}^{-1} b(y,z) f(x-y)g(x-y-z) \;dydz
    \end{equation*}
    so the result follows immediately from Minkowski's inequality and the translation invariance of the Lebesgue space norms.
\end{proof}

One important class of symbols $b$ satisfying the hypotheses of~\Cref{lem:pseudo-prod-Holder} are symbols $b$ supported on the region $|\xi|,|\eta| \lesssim 2^k$ which satisfy the symbol bounds
\begin{equation}\label{eqn:b-symb-bdds}
    |\partial_\xi^a \partial_\eta^b b(\xi,\eta)| \lesssim_{a,b} 2^{-(a+b)k}
\end{equation}
since for these symbols,
\begin{equation}\label{eqn:F-inv-b-bd}
    |\mathcal{F}^{-1} b(x,y)| \lesssim \frac{2^k}{(1+|2^kx|)^N} \frac{2^k}{(1+|2^k y|)^N}
\end{equation}
In particular, these symbols have the following pseudolocality property:
\begin{lemma}\label{lem:pseudoloc-lem}
    Suppose $b_k$ is supported in the region $|\xi|,|\eta| \lesssim 2^k$ and satisfies the symbol bounds~\eqref{eqn:b-symb-bdds}.  If the distance between the region $\mathcal{R} \subset \bbR$ and $\supp f$ is $\sim 2^j$, then
    \begin{equation}\label{eqn:bilin-pseudoprod-pseudoloc}
        \lVert B(f,g)(x) \rVert_{L^p(\mathcal{R})} \lesssim_N \jBra{2^{j+k}}^{-N} \lVert f \rVert_{L^q} \lVert g \rVert_{L^r}
    \end{equation}
    where $N \geq 1$ is fixed and $\frac{1}{p} = \frac{1}{q} + \frac{1}{r}$.  The estimate~\eqref{eqn:bilin-pseudoprod-pseudoloc} also holds if $\mathcal{R}$ is separated from $\supp g$ by a distance $\sim 2^j$.
\end{lemma}
\begin{proof}
    We begin with the physical space representation of $B(f,g)$:
    \begin{equation*}
        B(f,g)(x) = 2\pi \iint \mathcal{F}^{-1}(y,z) f(x-y) g(x-y-z)\;dydz
    \end{equation*}
    If $x \in \mathcal{R}$, then $f(x-y) = 0$ for $|y| \ll 2^j$, so using the decay estimate~\eqref{eqn:F-inv-b-bd} gives
    \begin{equation*}\begin{split}
        \lVert B(f,g)(x) \rVert_{L^p(\mathcal{R})}  \leq& 2\pi \iint_{|y| \gtrsim 2^j} |\mathcal{F}^{-1} b(y,z)| \lVert f(x-y) g(x-y-z) \rVert_{L^p}\;dydz\\
                                                    \lesssim_N& \iint_{|y| \gtrsim 2^j} \frac{2^k}{\jBra{2^k y}^N} \frac{2^k}{\jBra{2^k z}^N}\;dydz \lVert f \rVert_{L^q} \lVert g \rVert_{L^r}\\
                                                    \lesssim_N& \jBra{2^{j+k}}^{1-N} \lVert f \rVert_{L^q} \lVert g \rVert_{L^r}
    \end{split}\end{equation*}
    which is~\eqref{eqn:bilin-pseudoprod-pseudoloc} (with $N$ replaced by $N-1$).  Modifying the argument by restricting the double integral to the region $|y+z| \gtrsim 2^j$ gives the result when $d(\mathcal{R}, \supp g) \sim 2^j$.
\end{proof}

One particular application of~\Cref{lem:pseudoloc-lem} that we will make frequent use of is the following: For $k > j$, if $b_k$ is symbol satisfying the hypotheses of~\Cref{lem:pseudoloc-lem}, then
\begin{equation*}\begin{split}
    \chi_j^+B_k(f,g)(x) =& \chi_j^+ B_k(\chi_{\sim j}^+f, \chi_{\sim j}^+g) + \chi_j^+ B_k(\chi_{\sim j}^+f, (1-\chi_{\sim j}^+)g)\\
    &+ \chi_j^+ B_k((1-\chi_{\sim j}^+)f, g)(x)\\
    =& \chi_j^+ B_k(\chi_{\sim j}^+f, \chi_{\sim j}^+g) (x) + O_N(\jBra{2^{k+j}}^N \lVert f \rVert_{L^\infty} \lVert g \rVert_{L^\infty})
\end{split}\end{equation*}
Heuristically speaking, this means we can distribute spatial localization onto the arguments of a pseudoproduct, up to a dyadic `blurring,' and the error will be lower order provided that $2^{j+k} \gg 1$.

We will also work with cubic and quartic pseudoproduct operators of the form
\begin{equation*}
    C(f,g,h) = \mathcal{F}^{-1} \iiint c(\xi,\eta,\sigma) \hat{f}(\xi - \eta) \hat{g}(\eta - \sigma) \hat{h}(\sigma)\; d\eta d\sigma
\end{equation*}
and
\begin{equation*}
    Q(f,g,h,k) = \mathcal{F}^{-1} \iiint q(\xi,\eta,\sigma,\rho) \hat{f}(\xi - \eta) \hat{g}(\eta - \sigma) \hat{h}(\sigma- \rho) \hat{k}(\rho) \; d\eta d\sigma d\rho
\end{equation*}
We remark that the conclusions of~\Cref{lem:pseudo-prod-Holder,lem:pseudoloc-lem} also hold for these higher degree pseudoproducts (for cubic pseudoproducts, cf.~\cite[Section 2.2]{stewartLongTimeDecay2025}).

\subsection{Bounds involving the Hilbert transform}\label{sec:H-bdds}
We have the following commutator bound:
\begin{lemma}\label{lem:H-comm-bd}
    For any $f \in L^1$, 
    \begin{equation}\label{eqn:loc-H-comm-bd}
        \lVert [\chi^+_j(x), H] f \rVert_{L^\infty} \lesssim 2^{-j} \lVert f \rVert_{L^1}
    \end{equation}

    Moreover, for $n \geq 1$,
    \begin{equation}\label{eqn:loc-H-comm-bd-deriv}
        \lVert \partial_x^n [\chi^+_j(x), H] f \rVert_{L^\infty} \lesssim_n 2^{-(n+1)j} \lVert f \rVert_{L^1}
    \end{equation}
\end{lemma}
\begin{proof}
    By the Mean Value Theorem,
    \begin{equation*}
        |[\chi^+_j(x), H] f(x)| = \left|\int \frac{\chi_j^+(x) - \chi_j^+(y)}{x-y} f(y)\;dy\right| \lesssim 2^{-j} \lVert f \rVert_{L^1}
    \end{equation*}
    which is~\eqref{eqn:loc-H-comm-bd}.  The argument for~\eqref{eqn:loc-H-comm-bd-deriv} is similar once we notice that for a smooth function $F$, Taylor's theorem gives us the expansion
    \begin{equation*}
        F(x) = \sum_{k=0}^n \frac{(x-y)^k}{k!} F^{(k)}(y) + \frac{(x-y)^{n+1}}{n!} \int_0^1 F^{(n+1)}(y + s(x-y)) (1-s)^n\;ds
    \end{equation*}
    so
    \begin{equation*}
        \left| \partial_x^n \frac{F(x) - F(y)}{x-y}\right| = \left| \int_0^1 F^{(n+1)}(y + s(x-y)) (1-s)^n\;ds\right| \leq \lVert F^{(n+1)} \rVert_{L^\infty}
    \end{equation*}
    Taking $F = \chi_j^+$ now gives~\eqref{eqn:loc-H-comm-bd-deriv}.
\end{proof}

\section{Linear estimates}\label{sec:lin-ests}

The linear Benjamin-Ono flow  satisfies the same $L^1 \to L^\infty$ bounds (in any reference frame) as the one dimensional Schr\"odinger operator:
\begin{equation}\label{eqn:BO-disp-basic}
    \lVert e^{t(\partial_x + H\partial_x^2)} \rVert_{L^1 \to L^\infty}= \lVert e^{t H\partial_x^2} \rVert_{L^1 \to L^\infty} \sim t^{-1/2}
\end{equation}

The following refined estimate, which captures the more rapid decay away from the microlocal propagation set, will be key to our analysis.
\begin{thm}\label{thm:ref-lin-est}
    Let $j$ be a nonnegative integer, and $\epsilon \in (0,1)$.  Define
    \begin{equation}\label{eqn:k-0-def}
        k_0 = -\frac{1-\epsilon}{2} j 
    \end{equation}
    Then, for any positive integer $a$ and any time $t > 0$,
    \begin{equation}\label{eqn:low-freq-left-waves}
        \left\lVert \chi_j^+(x) e^{t(\partial_x + H \partial_x^2)} P_{\leq k_0} \partial_x^a \chi_{\lesssim j}^+(x) \right\rVert_{L^1 \to L^\infty} \lesssim \min\left(2^{-(a+2)j}, t^{-(a+2)}\right)
    \end{equation}
    Similarly, for $k > k_0$, we get the bound
    \begin{equation}\label{eqn:dyadic-freq-left-waves}
        \left\lVert \chi_j^{+}(x) e^{t(\partial_x + H \partial_x^2)} P_k \partial_x^a \chi_{\lesssim j}^+(x) \right\rVert_{L^1 \to L^\infty} \lesssim_M \min\left(2^{-M/2j}, t^{-M/2} \jBra{2^k}^{-M/2}\right)
    \end{equation}
    for any positive integer $M$.    Moreover, if $\ell > j - 10$, then for $t > 2^{\ell + 10}$, the following bounds hold:
    \begin{equation}\label{eqn:low-freq-right-waves}
        \left\lVert \chi_j^{+}(x) e^{t(\partial_x + H \partial_x^2)} P_{\leq k_0} \partial_x^a \chi_\ell^{+}(x) \right\rVert_{L^1 \to L^\infty} \lesssim t^{-(a+2)}
    \end{equation}
    and for all $t > \frac{2^{\ell + 10}}{\jBra{2^k}},$
    \begin{equation}\label{eqn:dyadic-freq-right-waves}
        \left\lVert \chi_j^{+}(x) e^{t(\partial_x + H \partial_x^2)} P_k \partial_x^a \chi_\ell^{+}(x) \right\rVert_{L^1 \to L^\infty} \lesssim_M t^{-M/2} \jBra{2^k}^{-M/2}
    \end{equation}
    for any positive integer $M$.
\end{thm}
\begin{proof}
    We will give a complete proof for~\eqref{eqn:low-freq-left-waves}, and then explain how to modify the argument to prove~\cref{eqn:low-freq-right-waves,eqn:dyadic-freq-left-waves,eqn:dyadic-freq-right-waves}.  To begin, we observe that the operator in~\eqref{eqn:low-freq-left-waves} has an integral kernel given by
    \begin{equation}
        K(x,y) = i^a \frac{\chi_j^+(x) \chi_{\lesssim j}^+(y)}{2\pi} \int e^{i \phi} \chi_{\leq k_0}(\xi) \xi^a\;d\xi
    \end{equation}
    where
    \begin{equation}\label{eqn:phi-lin-def}
        \phi = \xi(x-y+t) + t|\xi|\xi
    \end{equation}
    The estimate~\eqref{eqn:low-freq-left-waves} is equivalent to showing that
    \begin{equation*}
        \lVert K(x,y) \rVert_{L^\infty_{x,y}} \lesssim 2^{-k_0}\min\left(2^{-(a+2)j}, t^{-(a+2)}\right)
    \end{equation*}
    which in turn is equivalent to bounding the integral expression. Now, on the support of $K(x,y)$, 
    \begin{equation}\label{eqn:phi-lin-nonstat}
        \partial_\xi \phi = (x-y+t) + 2 t |\xi| > 0
    \end{equation}
    so the phase $\phi$ is nonstationary, which (at least heuristically) should lead to good decay.  This is complicated somewhat by the lack of smoothness of $\phi$ (it is not $C^2$ in the $\xi$ variable), so we must proceed carefully.  

    For $a,b,c \geq 0$, we define
    \begin{equation*}
        I_{a,b,c} = \int e^{i\phi} \frac{\xi^a \partial_\xi^b \chi_{\leq k_0}(\xi)}{[\partial_\xi \phi]^c}\;d\xi
    \end{equation*}
    and
    \begin{equation*}
        J_{a,b,c} = \int e^{i\phi} \frac{\sgn(\xi)\xi^a \partial_\xi^b \chi_{\leq k_0}(\xi)}{[\partial_\xi \phi]^c}\;d\xi
    \end{equation*}
    Integrating by parts and observing that $\partial_\xi^2 \phi = 2t \sgn(\xi)$, we find that
    \begin{equation}\label{eqn:I-abc-recurrence}\begin{split}
        I_{a,b,c} =& \int \frac{1}{i\partial_\xi \phi} \partial_\xi e^{i\phi} \frac{\xi^a \partial_\xi^b \chi_{\leq k_0}(\xi)}{[\partial_\xi \phi]^c}\;d\xi\\
                    =& ai I_{a-1,b, c+1} + i I_{a, b+1, c+1} - 2(c+1)i tJ_{a,b,c+2}
    \end{split}\end{equation}
    For $J_{a,b,c}$, the presence of the nonsmooth function $\sgn(\xi)$ results in an additional boundary term for $a = 0$:
    \begin{equation}\label{eqn:J-abc-recurrence}
        J_{a,b,c} = \left. \frac{2i e^{i\phi} \xi^a\partial_\xi^b \chi_{\leq k_0}(\xi)}{[\partial_\xi \phi]^{c+1}}\right|_{\xi=0} + ai J_{a-1, b, c+1} + i J_{a,b+1, c+1} - t(c+1)i I_{a,b,c+2}
    \end{equation}
    By iterating the recursion relations $M \geq a + 2$ times, we see that
    \begin{equation}\label{eqn:I-a00-expr} \begin{split}
        I_{a,0,0} =& \sum_{m=1}^{M-a-1} C_{\text{boundary}, m}^a \frac{t^m}{(x-y+t)^{a+2+m}} +  \sum_{%
            \substack{\alpha + \beta + \gamma = M\\\alpha \geq 0, \gamma \text{ even}}%
            } t^\gamma C^a_{\alpha,\beta,\gamma} I_{a-\alpha, \beta, N+\gamma}\\
            &+ \sum_{%
            \substack{\alpha + \beta + \gamma = M\\\alpha \geq 0, \gamma \text{ even}}%
            } t^\gamma C^a_{\alpha,\beta,\gamma} J_{a-\alpha, \beta, N+\gamma}
    \end{split}\end{equation}
    for constants $C_{\text{boundary},m}^a$ and $C^a_{\alpha,\beta,\gamma}$ and with the convention that $I_{a,b,c} = J_{a,b,c} = 0$ for $a < 0$.  For $(x,y) \in \supp K(x,y)$, we have that 
    \begin{equation*}
        \left|\frac{t^m}{(x-y+t)^{a+2+m}}\right| \leq \left|\frac{t}{(x-y+t)^{a+3}}\right| \lesssim \min(t 2^{-(a+3)j}, t^{-(a+2)})
    \end{equation*}
    which is the bound required by~\eqref{eqn:low-freq-left-waves}.  Thus, it only remains to show that the remaining integral terms are lower order for $M$ sufficiently large.  Using Minkowski's inequality and observing that $\chi_{\leq k_0}$ is supported on an interval of length $\sim 2^{k_0}$, we find that
    \begin{equation*}
        |I_{a,b,c}| + |J_{a,b,c}| \lesssim \frac{2^{k_0(a-b+1)}}{(x-y+t)^c}
    \end{equation*}
    so each of the terms in the sum is bounded by
    \begin{equation*} \begin{split}
        t^\gamma \left(|I_{a-\alpha, \beta, M+\gamma}| + |J_{a-\alpha, \beta, M+\gamma}|\right) \lesssim& \frac{2^{k_0(a +1 - \alpha - \beta)} t^\gamma}{(x-y+t)^{M+\gamma}}\\
        \lesssim& 2^{(a+1 + \gamma - M) k_0} \min\left(t^\gamma 2^{-(M+\gamma)j}, t^{-M}\right)\\
        \lesssim& 2^{(a+1-M) k_0} \min\left(2^{-Mj}, t^{-M}\right)
    \end{split}\end{equation*}
    where to obtain the second line we have used the fact that $\alpha + \beta + \gamma = M$.  Now, by~\eqref{eqn:k-0-def}, we see that for $M > 2(a+1)$
    \begin{equation*}
        2^{(a+1-M) k_0} \leq 2^{-M k_0} \leq 2^{\frac{M}{2}j}
    \end{equation*}
    so
    \begin{equation*} \begin{split}
        t^\gamma \left(|I_{a-\alpha, \beta, N+\gamma}| + |J_{a-\alpha, \beta, N+\gamma}|\right) \lesssim& \min\left(2^{-\frac{M}{2}j}, t^{-\frac{M}{2}}\right)
    \end{split}\end{equation*}
    which is consistent with~\eqref{eqn:low-freq-left-waves} provided we take $N \geq 2a + 4$.

    The argument for~\eqref{eqn:low-freq-right-waves} is similar once we notice that $\partial_\xi \phi \gtrsim t$ for all the relevant values of $x,y$ and $t$.  For~\eqref{eqn:dyadic-freq-left-waves} and~\eqref{eqn:dyadic-freq-right-waves}, we simply substitute the bound
    \begin{equation*}
        \partial_\xi \phi \gtrsim (x-y) + t\jBra{2^k}
    \end{equation*}
    and repeat the analysis, observing that the frequency localization away from $0$ suppresses the boundary terms.
\end{proof}

Since
\begin{equation*}
    (\partial_x + H \partial_x^2) P^+_k = (\partial_x - i \partial_x^2) P^+_k
\end{equation*}
we immediately obtain the following corollary about the Schr\"odinger propagator with drift:
\begin{cor}\label{cor:Schro-prop}
    Let $j$ be a nonnegative integer, and $\epsilon \in (0,1)$.  Then, for any $k > k_0$ and any positive integer $M$, we have that for $t > 0$
    \begin{equation}\label{eqn:dyadic-freq-left-waves-Schro}
        \left\lVert \chi_j^{+}(x) e^{t(\partial_x - i \partial_x^2)} P_k^+ \chi_{\lesssim j}^+(x) \right\rVert_{L^1 \to L^\infty} \lesssim_M \min\left(2^{-M/2j}, t^{-M/2} \jBra{2^k}^{-M/2}\right)
    \end{equation}
    Moreover, if $\ell > j - 10$, then for $t > \frac{2^{\ell + 10}}{\jBra{2^k}}$,
    \begin{equation}\label{eqn:dyadic-freq-right-waves-Schro}
        \left\lVert \chi_j^{+}(x) e^{t(\partial_x - i \partial_x^2)} P_k^+ \chi_\ell^{+}(x) \right\rVert_{L^1 \to L^\infty} \lesssim_M t^{-M/2} \jBra{2^k}^{-M/2}
    \end{equation}
\end{cor}

\section{The normal form/approximate gauge transformation}\label{sec:NFGT}

We now show how to obtain an appropriate normal form transformation for the equation.  To begin, we project the solutions to positive frequencies of size $\sim 2^k$:
\begin{equation*}
    \left(\partial_t + i \partial_x^2\right) u_k^+ = -P_k^+ \partial_x \left(u^2\right)
\end{equation*}
Note that we could also analogously define $u_k^-$ be projecting to negative frequencies of size $\sim 2^k$.  However, since we are interested in real-valued solutions, this is unnecessary ($u$ can be reconstructed based only on $\hat{u}(\xi)$ for $\xi > 0$), so we will omit them.  Rearranging the previous equation yields the Schr\"odinger-type equation
\begin{equation*}
    \left( i\partial_t - \partial_x^2 \right) u_k^+ = -i P_k^+ \partial_x \left( u \right)^2
\end{equation*}
We also define the function
\begin{equation*}
    \Phi(x,t) = -\int_x^\infty u(x,t)\;dx
\end{equation*}
and note that $\Phi$ satisfies the equation
\begin{equation}\label{eqn:Phi-BO}
    \partial_t \Phi - H \partial_x^2 \Phi = -\left(\partial_x \Phi\right)^2
\end{equation}
With these ingredients in hand, we define the $N$th order normal form/gauge transformation $u_k^+ \mapsto v_k^+$ to be
\begin{equation}\label{eqn:NFGT-def}
    v_{k}^+ = v_{k,N}^+ := \left(u_k^+ + B_{k,N}^+(u,u)\right)E_N(\Phi_{\ll_N k})
\end{equation}
where 
\begin{equation*}
    E_N(x) = \sum_{n=0}^N \frac{(-ix)^n}{n!}
\end{equation*}
is the $N$th order Taylor polynomial for $\exp(-ix)$ and $B_{k,N}^+$ is a bilinear pseudoproduct that we will choose to approximately cancel the quadratic terms.  To determine $B_{k,N}^+$, we first observe that
\begin{equation}\label{eqn:NG-trans-raw}\begin{split}
    (i\partial_t - \partial_x^2) v_k^+ =& [(\partial_t + i\partial_x^2) \Phi]_{\ll_N k} \left(u_k^+ + B_{k,N}^+(u,u)\right)E_{N-1}(\Phi_{\ll_N k})\\
                                        &+ 2i \partial_x \Phi_{\ll_N k} \partial_x \left(u_k^+ + B_{k,N}^+(u,u)\right)E_{N-1}(\Phi_{\ll_N k})\\
                                        &- iP_k^+ \partial_x(u^2) E_N(\Phi_{\ll_N k})\\
                                        &+ \left((i\partial_t - \partial_x^2) B_{k,N}^+(u,u)\right) E_N(\Phi_{\ll_N k})
\end{split}
\end{equation}
For the first term, we use~\eqref{eqn:Phi-BO} and the fact that $\partial_x \Phi = u$ to conclude that
\begin{equation}\label{eqn:NG-trans-1-1}\begin{split}
    (\partial_t + i \partial_x^2) \Phi_{\ll_N k} =& \left((\partial_t - H \partial_x^2) \Phi\right)_{\ll_N k} + (H+ i) \partial_x^2 \Phi_{\ll_N k}\\
                                            =& -\left(u^2\right)_{\ll_N k} + (H+i) \partial_x u_{\ll_N k}
\end{split}\end{equation}
Moreover, for the pseudoproduct involving $B_{k,N}^+(u,u)$, we see that
\begin{equation}\label{eqn:NG-trans-1-2}\begin{split}
    (i\partial_t - \partial_x^2) B_{k,N}^+(u,u) =& B_{k,N}^+((i\partial_t - \partial_x^2)u,u) + B_{k,N}^+(u,(i\partial_t - \partial_x^2)u) \\&- 2 B_{k,N}^+(\partial_x u, \partial_x u)\\
    =& -iB_{k,N}^+(\partial_x(u^2), u) - iB_{k,N}^+(u, \partial_x (u^2)) \\
    &+ iB_{k,N}^+((H+i)\partial_x^2 u, u) + iB_{k,N}^+(u, (H+i)\partial_x^2 u) \\
    &- 2 B_{k,N}^+(\partial_x u, \partial_x u)
\end{split}\end{equation}
Inserting~\eqref{eqn:NG-trans-1-1} and~\eqref{eqn:NG-trans-1-2} into~\eqref{eqn:NG-trans-raw} and rearranging, we obtain
\begin{equation}\label{eqn:NG-trans-ref}\begin{split}
    (i\partial_t - \partial_x^2) v_k^+ =& \mathcal{B}_{k,N}^+(u,u) E_{N-1}(\Phi_{\ll_N k})\\
    &+ \left(\mathcal{B}_{k,N}^+(u,u) + \mathcal{B}_{k,N}^{+,\text{rem}}(u,u)\right) \frac{\left(-i \Phi_{\ll_N k}\right)^N}{N!} \\
    &+ C_{k,N}^+(u,u,u) E_{N-1}(\Phi_{\ll_N k}) + \tilde{C}_{k,N}^+(u,u,u) \frac{\left(-i \Phi_{\ll_N k}\right)^N}{N!}\\
    &+ Q_{k,N}^+(u,u,u,u) E_{N-1}(\Phi_{\ll_N k})
\end{split}
\end{equation}
where
\begin{align}\begin{split}
    \mathcal{B}_{k,N}^+(u,u) =& -i P_k^+\partial_x(u^2) + (H+i)\partial_x u_{\ll_N k} u_k^+ + 2i u_{\ll_N k} \partial_x u_k^+ \\
    & + i B_{k,N}^+((H+i)\partial_x^2 u, u)+ i B_{k,N}^+(u, (H+i)\partial_x^2 u)\\
    &- 2 B_{k,N}^+(\partial_x u, \partial_x u)
\end{split}\label{eqn:calB-p-k-def}\\[2ex]
    \mathcal{B}_{k,N}^{+,\text{rem}}(u,u) =& (H+i)\partial_x u_{\ll_N k} u_k^+ + 2iu_{\ll_N k} \partial_x u_k^+ \label{eqn:calB-p-k-rem-def}\\[2ex]
    \tilde{C}_{k,N}^+(u,u,u) =& -iB_{k,N}^+(\partial_x(u^2), u) - iB_{k,N}^+(u, \partial_x (u^2))\label{eqn:C-tilde-p-k-def}\\[2ex]
    \begin{split}
    C_{k,N}^+(u,u,u) =& \tilde{C}_{k,N}^+(u,u,u) - (u^2)_{\ll_N k} u_k^+\\
    &+ \left(2i u_{\ll_N k} \cdot \partial_x + (H+i) \partial_x u_{\ll_N k}\right) B_{k,N}^+(u,u)
    \end{split}\label{eqn:C-p-k-def}\\[2ex]
    Q_{k,N}^+(u,u,u,u) =& -(u^2)_{\ll_N k} B_{k,N}^+(u,u)\label{eqn:Q-p-k-def}
\end{align}
We will choose $B^+_{k,N}$ so that $\mathcal{B}^+_{k,N}(u,u) = 0$.  Taking the Fourier transform and recalling that
\begin{equation}\label{eqn:Hpmi-ident}\mathcal{F} 
    (H\pm i) = i(1 \mp \sgn(\xi)) = 2i \bbOne_{\mp \xi > 0}
\end{equation}
we see that the only choice for $b_{k,N}^+$ that is symmetric when we interchange the roles of $\eta$ and $\xi - \eta$ is
\begin{equation}\label{eqn:NF-symb}\begin{split}
    2\left((\xi-\eta)_-^2 + \eta_{-}^2 + (\xi-\eta)\eta\right)b_{k,N}^+ =& \chi_k^+(\xi) \xi - \chi_k^+(\xi-\eta)\chi_{\ll_N k}(\eta) \left((\xi-\eta) + \eta_- \right)\\
    &- \chi_k^+(\eta)\chi_{\ll_N k}(\xi-\eta) \left(\eta + (\xi-\eta)_{-}\right)
\end{split}\end{equation}
where $x_- = \frac{1}{2}(|x| - x)$ denotes the negative part of $x$, and for notational simplicity we adopt the convention that $x_-^2 = (x_-)^2$.  

Although the symbol~\eqref{eqn:NF-symb} formally defines the quadratic correction for the normal form transformation, there are two concerns that need to be addressed.  First, it is not clear as written that the symbol is nonsingular on $(\xi-\eta)_-^2 + \eta_{-}^2 + (\xi-\eta)\eta = 0$.  Second, the symbol is not smooth, which presents problems if we want to apply the $L^p$ theory developed in~\Cref{sec:pseudo-prod-ops}.  To address both of the difficulties, we will show that there are smooth, nonsingular symbols $b_k^{+,\pm\pm\pm}(\xi,\eta)$ supported in the region $\{\xi \sim 2^k\}$ with the property that
\begin{equation*}
    b_{k,N}^+(\xi,\eta) = \sum_{e_1,e_2,e_3 \in \{+,-\}} \bbOne_{e_1\xi > 0} \bbOne_{e_2(\xi-\eta) > 0} \bbOne_{e_3\eta > 0} b_{k,N}^{+,e_1e_1e_3}(\xi,\eta)
\end{equation*}
We will also show that all of the symbols $b^{+,-e_2e_3}_{k,N}$ can be chosen to be zero.  Moreover, by the support condition,
$$|\xi-\eta| + |\eta| \gtrsim 2^k$$
so at most one of the frequencies $\xi-\eta$ and $\eta$ can vanish.  Since $(H \pm i) P^\pm_{\gtrsim k}$ has a smooth symbol, this allows us to write 
\begin{equation}\label{eqn:B-k-N-decomp}\begin{split}
    B_{k,N}^+(u,v)  =&\!\!\!\! \sum_{e_i \in \{+,-\}}\!\!\!\! P^{\pm e_1}B_{k,N}^{+,e_1e_1e_3}(P^{\pm e_2}u,P^{\pm e_3}v)\\
                =&\!\!\!\!\!\! \sum_{e_2,e_3 \in \{+,-\}}\!\!\!\!\!\!  B_{k,N,1}^{+,+e_2e_3}(u,P^{\pm e_3}_{\lesssim k}v) +  B_{k,N,2}^{+,+e_2e_3}(P^{\pm e_2}_{\lesssim k}u,v) \\
                &\qquad\qquad+ B_{k,N,3}^{+,+e_2e_3}(u,v)
\end{split}\end{equation}
where the second line expresses the fact that (1) there is no Hilbert transform applied to the bilinear pseudoproducts $B^{+,+e_2,e_3}_{k,N,j}(u,u)$, and (2) at most one of the input functions contains a Hilbert transform.  We will provide a proof of this decomposition in~\Cref{sec:b-k-N-decomp}.  We will then finish this section by proving decay estimates for the pseudoproduct operators $B_{k,N}^+(w,w)$, $\mathcal{B}_{k,N}^{+,\text{rem}}(w,w)$, $Q_{k,N}^+(w,w,w,w)$, and $C_{k,N}^+(w,w,w)$, where 
\begin{equation}\label{eqn:w-def}
    w(x,t) = u(x+t,t)
\end{equation}
is the expression for the solution in a moving reference frame.

\subsection{Expressions for the \texorpdfstring{$b_{k,N}^{+,e_1e_2e_3}$}{b\_(b,k)\^(+, e\_1e\_2e\_3}}\label{sec:b-k-N-decomp}

Here, we give expressions for the symbols $b_{k,N}^{+,e_1e_2e_3}$ appearing in~\eqref{eqn:B-k-N-decomp}, and show that they obey the required conditions.

\subsubsection{The symbol $b_k^{+,+++}$:}
The symbol $b_{k,N}^{+,+++}$ is given by
\begin{equation}\label{eqn:b-k-p-ppp}\begin{split}
    b_{k,N}^{+,+++}(\xi,\eta) =& \frac{\chi_k^+(\xi)\xi - \chi_k^+(\xi-\eta)\chi_{\ll_N k}(\eta)(\xi-\eta) - \chi_k^+(\eta) \chi_{\ll_N k}(\xi - \eta) \eta}{2(\xi-\eta) \eta}\\
    =& \frac{\chi_k^+(\xi) - \chi_{k}^+(\xi-\eta) \chi_{\ll_N k}(\eta)}{2\eta} + \frac{\chi_k^+(\xi) - \chi_{k}^+(\eta) \chi_{\ll_N k}(\xi-\eta)}{2(\xi-\eta)}\\
    =& \chi_{\ll_N k}(\eta) \frac{\chi_k^+(\xi) + \chi_k^+(\xi-\eta)}{2\eta} + \chi_k^+(\xi) \frac{\chi_{\gtrsim_N k}(\eta)}{2\eta} \\
    &+\chi_{\ll_N k}(\xi-\eta) \frac{\chi_k^+(\xi) - \chi_k^+(\eta)}{2(\xi-\eta)} + \chi_k^+(\xi) \frac{\chi_{\gtrsim_N k}(\xi-\eta)}{2(\xi-\eta)}
\end{split}\end{equation}
which by the mean value theorem is smooth and nonsingular, supported in the region $|\xi|+|\eta| \gtrsim 2^{k}$, and satisfies the symbol-type bounds
\begin{equation*}
    |\partial_\xi^a \partial_\eta^b b_{k,N}^{+,+++}(\xi,\eta)| \lesssim_{a,b,N} (|\xi|+|\eta|)^{-a-b-1}
\end{equation*}

\subsubsection{The symbols $b_{k,N}^{+,++-}$ and \textbf{$b_{k,N}^{+,+-+}$}}
Here, we have
\begin{equation}\label{eqn:b-k-p-ppm}\begin{split}
    b_{k,N}^{+,++-}(\xi,\eta)   =& \frac{\chi_k^+(\xi) - \chi_k^+(\xi-\eta) \chi_{\ll_N k}(\eta)}{2\eta}\\
                            =& \chi_k^+(\xi) \frac{\chi_{\gtrsim_N k}(\eta)}{2\eta} + \chi_{\ll_N k}(\eta) \frac{\chi_k^+(\xi) - \chi_k^+(\xi-\eta)}{2\eta}
\end{split}
\end{equation}
which is smooth, nonsingular, and supported in the region $|\xi| + |\eta| \gtrsim 2^k$.  By symmetry,
\begin{equation}\label{eqn:b-k-p-pmp}
    b_{k,N}^{+,+-+}(\xi,\eta) = b_{k,N}^{+,++-}(\xi,\xi-\eta)
\end{equation}
also satisfies the same bounds.

\subsubsection{The symbols $b_{k,N}^{+,+--}$ and $b_{k,N}^{+,-++}$}
Since it is not possible for two negative frequencies to sum to a positive frequency or for two positive frequencies to sum to a negative frequency, we may take
\begin{equation}\label{eqn:b-k-pmm-mpp}
    b_{k,N}^{+,+--}(\xi,\eta) = b_{k,N}^{+,-++}(\xi,\eta) = 0
\end{equation}

\subsubsection{The symbols $b_{k,N}^{+,---}$, $b_{k,N}^{+,-+-}$ and $b_{k,N}^{+,--+}$}
We will show that all of the $b_{k,N}^{+,---}$, $b_{k,N}^{+,-+-}$ and $b_{k,N}^{+,--+}$ vanish, as required in~\eqref{eqn:B-k-N-decomp}.  We begin by defining $b_{k,N}^{+,---}$.  From~\eqref{eqn:NF-symb}, we see that
\begin{equation}\label{eqn:b-k-mmm}
    b_{k,N}^{+,---}(\xi,\eta) = 0
\end{equation}
For $b_{k,N}^{+,-+-}$, naively inserting the condition $\xi < 0$, $\xi - \eta > 0$, $\eta < 0$ into~\eqref{eqn:NF-symb} gives
\begin{equation*}
    -\frac{\chi_k^+(\xi-\eta) \chi_{\ll_N k}(\eta)}{2\eta}
\end{equation*}
This is both nonzero and singular, and thus clearly unacceptable.  A more careful analysis shows that in the region $\xi-\eta \sim 2^k$ and $\xi, \eta < 0$, the function $\chi_k^+(\xi-\eta) \chi_{\ll_N k}(\eta)$ vanishes equivalently, so we may take
\begin{equation}\label{eqn:b-k-p-mpm}
    b_{k,N}^{+,-+-}(\xi,\eta) = 0
\end{equation}
By symmetry, this also gives
\begin{equation}\label{eqn:b-k-p-mmp}
    b_{k,N}^{+,--+}(\xi,\eta) = b_{k,N}^{+,-+-}(\xi,\xi-\eta) = 0
\end{equation}

\subsection{Decay for \texorpdfstring{$B_{k,N}^+(w,w)$}{B\_(k,n)\^+(w,w)}}\label{sec:B-k-decay}

Using the expression~\eqref{eqn:B-k-N-decomp} derived in the previous section, we will now obtain spatial decay bounds for the bilinear expression $B_{k,N}^+(w,w)$ under the decay hypothesis~\eqref{eqn:main-decay-hypo}.  We first observe that~\eqref{eqn:main-decay-hypo} is equivalent to the statement that
\begin{equation*}
    \lVert \chi^\pm_j w \rVert_{L^\infty} \lesssim 2^{-(1+\epsilon)j}
\end{equation*}
By symmetry, $B_{k,N}^+(w,w)$ can be written as a sum of order $-1$ bilinear pseudoproducts of the form
$$\tilde{B}_{k,N}^+(w, P^\mp_{\lesssim k}w)$$
or
$$\tilde{B}_{k,N}^+(w, w)$$
We will give decay estimates for the first type of terms -- the second type, which does not contain Hilbert transforms, has better spatial decay.  We have that
\begin{subequations}\begin{align}
    \lVert \chi^\pm_j(x) \tilde{B}_{k,N}^+(w, P^\mp_{\lesssim k} w) \rVert_{L^\infty} \!\!\leq\!\!\!\!\!\!\!\!\!\!\!\!\!&\qquad\; \lVert \chi^\pm_j(x) \tilde{B}_{k,N}^+\bigl(\chi^\pm_{\sim j}(x) w, P_{\lesssim k}\chi^\pm_{\sim j}(x)P^{\mp} w\bigr) \rVert_{L^\infty}\label{eqn:B-k-p-space-bd-main}\\
    &+\! \lVert \chi^\pm_j(x) \tilde{B}_{k,N}^+\bigl(\bigl[1-\chi^\pm_{\sim j}(x)\bigr] w, P^{\mp} w_{\lesssim k}\bigr) \rVert_{L^\infty}\label{eqn:B-k-p-space-bd-err1}\\
    &+\! \lVert \chi^\pm_j(x) \tilde{B}_{k,N}^+\bigl(\chi^\pm_j(x) w, P_{\lesssim k}\bigl[1-\chi^\pm_{\sim j}(x)\bigr] P^{\mp} w\bigr) \rVert_{L^\infty}\label{eqn:B-k-p-space-bd-err2}
\end{align}\end{subequations}
For the second term, by the pseudolocality principle we have that
\begin{equation*}\begin{split}
    \eqref{eqn:B-k-p-space-bd-err1} \lesssim_M& 2^{-M(j+k)} \lVert \tilde{B}_k^+\left(\left[1-\chi^\pm_{\sim j}(x)\right] w, P^\mp_{\lesssim k} w\right) \rVert_{L^\infty_x}\\
    \lesssim_M& 2^{-M(j+k)} 2^{-k} \lVert w \rVert_{L^\infty} \lVert P^\mp_{\lesssim k}  w \rVert_{L^\infty}\\
    \lesssim_M& 2^{-M(j+k)} 2^{-k}  2^{k/2} \lVert P^\mp_{\lesssim k} w \rVert_{L^2}\\
    \lesssim_M& 2^{-M(j+k)} 2^{-k/2}
\end{split}\end{equation*}
In particular, for $k > -\frac{1-\epsilon}{2} j$, these terms decay very rapidly in $j$.  A similar argument applies to the last term~\eqref{eqn:B-k-p-space-bd-err2}.  Turning to the main term~\eqref{eqn:B-k-p-space-bd-main}, we see that
\begin{equation*}\begin{split}
    \eqref{eqn:B-k-p-space-bd-main} \lesssim& 2^{-k} \lVert \chi^\pm_{\sim j}(x) w \rVert_{L^\infty} \lVert P_{\lesssim k} \chi_j^\pm(x) P^\mp w \rVert_{L^\infty}\\
    \lesssim& 2^{-(2+\epsilon)j} 2^{-k} \jBra{2^{\epsilon k}}
\end{split}
\end{equation*}
where on the last line we have used the commutator bound from~\Cref{lem:H-comm-bd} to estimate
\begin{equation*}\begin{split}
    \lVert P_{\lesssim k} \chi_j^\pm(x) P^\mp w \rVert_{L^\infty} \lesssim& \lVert P_{\lesssim k} [\chi_j^\pm(x), H] w \rVert_{L^\infty} + \lVert P_{\lesssim k}^\mp   \chi_j^\pm(x)  w \rVert_{L^\infty}\\
    \lesssim& 2^{-j} + 2^{k/q} \lVert \chi_j^\pm(x)  w  \rVert_{L^{q}}\\
    \lesssim& 2^{-j} \jBra{2^{\min(\epsilon, 1/100) k}}
\end{split}\end{equation*}
with the choice $q = \max(\epsilon^{-1}, 100)$ giving the last line.  It follows that for $k > -\frac{1-\epsilon}{2}j$, we have the bound
\begin{equation}\label{eqn:B-k-bd}
    \lVert \chi_j^+(x) B_{k,N}^+(w,w) \rVert_{L^\infty} \lesssim 2^{-(2+\epsilon)j} 2^{-k} \jBra{2^{\min(\epsilon, 1/100) k}}
\end{equation}

\subsection{Decay for \texorpdfstring{$\mathcal{B}_{k,N}^{+,\text{rem}}(w,w)$}{B\_(k,n)\^(+,rem)(w,w)}}\label{sec:B-k-rem-decay}

We now give a decay estimate for the quadratic terms $\mathcal{B}_{k,N}^{+,\text{rem}}(w,w)$ that were not removed by the normal form.  In contrast to the other terms in this section, the remainder terms contain derivatives and require us to use the condition $\lVert w \rVert_{L^\infty_t \dot{B}^{1/2}_{2,1}} < \infty$, so we will estimate these terms in $L^2$ spaces.  By~\eqref{eqn:calB-p-k-rem-def}, it suffices to prove bounds for
\begin{equation}\label{eqn:calB-rem-1}
	\lVert \chi_j^+(x) w_{\ll_N k} \partial_x w_k^+ \rVert_{L^2}
\end{equation}
and
\begin{equation}\label{eqn:calB-rem-2}
	\lVert \chi_j^+(x) \left(P^- \partial_x w_{\ll_N k}\right) w_k^+ \rVert_{L^2}
\end{equation}
For~\eqref{eqn:calB-rem-1}, we estimate
\begin{equation*}\begin{split}
	\eqref{eqn:calB-rem-1} 	\lesssim& \lVert \chi_j^+(x) w_{\ll_N k} \rVert_{L^\infty} \lVert  \partial_x w_k^+ \rVert_{L^2}\\
					\lesssim& 2^{-(1+\epsilon)j} 2^{k/2} c_k
\end{split}\end{equation*}
where we recall that $c_k$ represents a generic sequence in $\ell^1_k$.  Similarly, since the Hilbert transform is bounded on $L^2$, we have that
\begin{equation*}\begin{split}
	\eqref{eqn:calB-rem-2} 	\lesssim& \lVert \chi_j^+(x) w_{k}^+ \rVert_{L^\infty} \lVert  P^- \partial_x w_{\ll_N k} \rVert_{L^2}\\
					\lesssim& 2^{-(1+\epsilon)j} 2^{k/2} c_k
\end{split}\end{equation*}
so
\begin{equation}\label{eqn:calB-rem-bd}
	\lVert \chi_j^+(x)\mathcal{B}_{k,N}^{+,\text{rem}} (w,w) \rVert_{L^2} \lesssim 2^{-(1+\epsilon)j} 2^{k/2} c_k
\end{equation}
For later use, we also note
\begin{equation}\label{eqn:calB-rem-bd-L1}
    \lVert \mathcal{B}_{k,N}^{+,\text{rem}}(w,w) \rVert_{L^1} \lesssim 2^{k/2} c_k
\end{equation}
which follows immediately from the previous analysis since $w \in L^\infty_t L^2_x$.

\subsection{Quartic term bounds}

For the quartic terms, we take advantage of the bound~\eqref{eqn:B-k-bd} for $B_{k,N}^+ (w,w)$ and pseudolocality for $P_{\ll_N k}$ to obtain the bound
\begin{equation}\label{eqn:Q-k-bd}\begin{split}
    \lVert \chi_j^+ Q_k^+(w,w,w,w) \rVert_{L^\infty} \lesssim& \left\lVert \left(\chi_{\sim j}^+(x) w\right)_{\ll_N k} \chi_j^+ B_{k,N}^+(w,w) \right\rVert_{L^\infty_x} + \{ \textup{better}\}\\
    \lesssim& 2^{-(4+3\epsilon)j} 2^{-k}\jBra{2^{\min(\epsilon, 1/100) k}}
\end{split}\end{equation}

\subsection{Cubic term bounds}

We will show how to bound $C_{k,N}^+(w,w,w)$ by bounding each term in~\eqref{eqn:C-p-k-def}.  Since this includes the terms in $\tilde{C}_{k,N}^+(w,w,w)$, we will also obtain a bound for this expression. 
 Substituting~\eqref{eqn:C-tilde-p-k-def} into~\eqref{eqn:C-p-k-def}, we see that $C_{k,N}^+(w,w,w)$ is given by the expression
\begin{equation*}\begin{split}
    C_{k,N}^+(w, w, w) =& -(w^2)_{\ll_N k} w_k^+ - P^- \partial_x w_{\ll_N k} B_{k,N}^+(w, w)+ 2i w_{\ll_N k} \partial_x B_{k,N}^+(w, w)\\
    &- i B_{k,N}^+(\partial_x(w^2), w) - i B_{k,N}^+(w, \partial_x (w^2))\\
    =:& C_{k,N}^{+,1}(w,w,w) + C_{k,N}^{+,2}(w,w,w) + C_{k,N}^{+,3}(w,w,w)\\
    &+ C_{k,N}^{+,4}(w,w,w) + C_{k,N}^{+,5}(w,w,w)
\end{split}\end{equation*}
We will show how to bound each of these terms in the region $x \sim +2^j$.  For $C_{k,N}^{+,1}$, by pseudolocality,
\begin{equation}\label{eqn:C-k-p1-bd}\begin{split}
    \lVert \chi_j^+(x) C_{k,N}^{+,1}(w,w,w) \rVert_{L^\infty} \lesssim& \lVert (\chi_{\sim j}^+(x) w^2)_{\ll_N k} (\chi_{\sim j}^+(x) w)_k^+\rVert_{L^\infty} + \{\textup{better}\}\\
    \lesssim& 2^{-(3+3\epsilon)j}
\end{split}\end{equation}
For $C_{k,N}^{+,2}$, the main idea is to commute the $\chi_{\sim j}^+$ multiplier through the other terms.  However, this must be done carefully, as the Hilbert transform is not bounded on $L^\infty$.  To deal with this, we use the operator identity
\begin{equation}\label{eqn:4-op-ID}\begin{split}
    P_{\leq k} \chi^+_j H \partial_x = P_{\leq k} \left\{H (\partial_x \chi_j^+ + [\chi^+_j, \partial_x]) + \partial_x[\chi_j^+,H] + [H,[\partial_x, \chi^+_j]]\right\}
\end{split}\end{equation}
Note that the operator $[\partial_x, \chi^+_j]$ corresponds to multiplying by the function $\partial_x \chi^+_j(x)$, so
\begin{equation*}
    \lVert [H, [\partial_x, \chi^+_j]] \rVert_{L^1 \to L^\infty} \lesssim 2^{-2j}
\end{equation*}
by~\Cref{lem:H-comm-bd}.  Combining the identity~\eqref{eqn:4-op-ID} with pseudolocality for $P_{\ll_N k}$ now gives
\begin{equation*}\begin{split}
    \lVert \chi_j^+(x) \partial_x P^- P_{\ll_N k} w \rVert_{L^\infty} \leq& \lVert P_{\ll_N k} \chi_{\sim j}^+(x) \partial_x P^-  w \rVert_{L^\infty} + \{\textup{better}\}\\
    \leq& \left\lVert P_{\ll_N k}H (\partial_x \chi_{\sim j}^+ + [\chi^+_{\sim j}, \partial_x]) \right\rVert_{L^\infty}\\
    &+ \left\lVert \partial_x[\chi_{\sim j}^+,H] + [H,[\partial_x, \chi^+_{\sim j}]] w \right \rVert_{L^\infty} + \{\textup{better}\}\\
    \lesssim& 2^{(1+ 1/q)k} \lVert \chi_{\sim j} w \rVert_{L^q} + 2^{k/q} \lVert [\chi_{\sim j},\partial_x] w \rVert_{L^q} \\
    &+ \lVert \partial_x [\chi_{\sim j}^+,H] w \rVert_{L^\infty} + \lVert [H,[\partial_x, \chi_{\sim j}^+]] w \rVert_{L^\infty}\\
    \lesssim& 2^{k/q} \left(2^k + 2^{-j}\right) 2^{-(1+\epsilon-1/q)j} + 2^{-2j}
\end{split}\end{equation*}
for $q \in (1,\infty)$.  Choosing $1/q = \min(1/100, \epsilon) \in (0,1)$ and observing that 
\begin{equation*}
    2^{(1+1/q)k} \geq 2^{(1+1/q)\frac{-1+\epsilon}{2} j} \gg 2^{-j}
\end{equation*} 
we find that
\begin{equation}\label{eqn:4-op-lin}\begin{split}
    \lVert \chi_j^+(x) \partial_x (H+i) P_{\ll k} w \rVert_{L^\infty} 
    \lesssim& 2^{(1+\min(1/100, \epsilon))k} 2^{-j}
\end{split}\end{equation}
which gives the bound
\begin{equation}\label{eqn:C-k-p2-bd}
    \lVert \chi_j^+ C_{k,N}^{+,2}(w,w,w) \rVert_{L^\infty} \lesssim \jBra{2^{2\min(1/100, \epsilon) k}} 2^{-(3+\epsilon)j}
\end{equation}
The argument for $C_{k,N}^{+,3}$ is simpler: there is no Hilbert transform to contend with, and the frequency localization of $B_{k,N}^+(w,w)$ guarantees us that $${\lVert \chi_j^+(x) \partial_x B_{k,N}^+(w,w) \rVert_{L^\infty} \lesssim \jBra{2^{\min(1/100, \epsilon) k}} 2^{-(2+\epsilon)j}}$$ giving us the bound
\begin{equation}\label{eqn:C-k-p3-bd}
    \lVert \chi_j^+(x) C_{k,N}^{+,3}(w,w,w) \rVert_{L^\infty} \lesssim \jBra{2^{\min(1/100, \epsilon) k}} 2^{-(3+2\epsilon)j}
\end{equation}
For $C_{k,N}^{+,4}$, we expand the pseudoproduct $B_{k,N}^+(\partial_x(w)^2,w)$ to write
\begin{equation*}\begin{split}
    C_{k,N}^{+,4}(w,w,w) = i \tilde{B}_{k,N}^+(P_{\lesssim k}^\mp \partial_x(w^2),w) + i \tilde{B}_{k,N}^+(\partial_x (w^2), P_{\lesssim k}^\mp w)+ \{\text{easier}\}
\end{split}\end{equation*}
where $\{\text{easier}\}$ denotes pseudoproducts with no Hilbert transforms, with are simpler to handle.  Noting that $\tilde{B}_{k,N}^+(\partial_x f, g)$ is an order-$0$ pseudoproduct operator, pseudolocality gives the bound
\begin{equation*}
    \lVert \chi_j^+(x) \tilde{B}_{k,N}^+(\partial_x (w^2), P_{\lesssim k}^\mp w) \rVert_{L^\infty} \lesssim \lVert \chi_{\sim j}^+(x) w^2 \rVert_{L^\infty} \lVert \chi_{\sim j}^+(x) P_{\lesssim k}^\mp w \rVert_{L^\infty}
\end{equation*}
and
\begin{equation*}\begin{split}
    \lVert \chi_{\sim j}^+(x) P_{\lesssim k}^\mp w \rVert_{L^\infty} \leq& \lVert P_{\lesssim k} \chi_{[j-20, j+20]}^+(x) P^\mp w \rVert_{L^\infty} + \{\textup{better}\}\\
    \lesssim& \lVert [\chi_{\sim j}^+(x), H] w \rVert_{L^\infty} + 2^{k/q} \lVert P_{\lesssim k}^\mp \chi_{\sim j}^+ w \rVert_{L^{q}}\\
    &+ \{\textup{better}\}\\
    \lesssim& \jBra{2^{\min(1/100, \epsilon) k}} 2^{-j}
\end{split}\end{equation*}
so
\begin{equation}\label{eqn:C-p4-k-1-bd}
    \lVert \chi_j^+(x) \tilde{B}_{k,N}^+(\partial_x (w^2), P_{\lesssim k}^\mp w) \rVert_{L^\infty} \lesssim 2^{-(3+2\epsilon)j}\jBra{2^{\min(1/100, \epsilon) k}}
\end{equation}
For $\tilde{B}_{k,N}(P_{\lesssim k}^\mp \partial_x (w^2), w)$, pseudolocality again gives the bound
\begin{equation*}\begin{split}
    \lVert \chi_j^+(x) \tilde{B}_{k,N}(P_{\lesssim k}^\mp \partial_x (w^2), w) \rVert_{L^\infty} \lesssim& 2^{-k} \lVert \chi_{\sim j}^+(x) P_{\lesssim k}^\mp \partial_x (w^2) \rVert_{L^\infty} \lVert \chi_{\sim j}^+(x)  w \rVert_{L^\infty}\\
    &+ \{\textup{better}\}
\end{split}\end{equation*}
By using the operator identity~\eqref{eqn:4-op-ID} and arguing as in the proof of~\eqref{eqn:4-op-lin}, we find that
\begin{equation*}
    \lVert \chi_{\sim j}^+(x) P_{\lesssim k}^\mp \partial_x (w^2) \rVert_{L^\infty} \lesssim 2^{k} \jBra{2^{\min(1/100, \epsilon) k}} 2^{-(2+\epsilon)j} + 2^{-2j}
\end{equation*}
so
\begin{equation}\label{eqn:C-p4-k-2-bd}
    \lVert \chi_j^+(x) \tilde{B}_{k,N}(P_{\lesssim k}^\mp \partial_x (w^2), w) \rVert_{L^\infty} \lesssim 2^{-(3+\epsilon) j} \left( \jBra{2^{\min(1/100, \epsilon) k}} 2^{-\epsilon j} + 2^{-k}\right)
\end{equation}
Combining~\eqref{eqn:C-p4-k-1-bd} and~\eqref{eqn:C-p4-k-2-bd}, we find that
\begin{equation}\label{eqn:C-k-p4-bd}
    \lVert \chi_j^+(x) C^{+,4}_{k,N}(w,w,w) \rVert_{L^\infty} \lesssim 2^{-(3+\epsilon) j} \left( \jBra{2^{\min(1/100, \epsilon) k}} 2^{-\epsilon j} + 2^{-k}\right)
\end{equation}
By symmetry, the bounds for $C^{+,5}_{k,N}(w,w,w)$ are identical.  

Combining the bounds~\cref{eqn:C-k-p1-bd,eqn:C-k-p3-bd,eqn:C-k-p2-bd,eqn:C-k-p4-bd}, we find that
\begin{equation}\label{eqn:C-k-bd}
    \lVert \chi_j^+(x) C^{+}_k(w,w,w) \rVert_{L^\infty} \lesssim 2^{-(3+\epsilon) j} \left( 2^{2\min(1/100, \epsilon) k} + 2^{-k}\right)
\end{equation}

\section{Proof of the main theorems}\label{sec:proof-sec}

We now show how the linear estimates in~\Cref{sec:lin-ests} can be combined with the normal form transformation given in~\Cref{sec:NFGT} to prove~\Cref{thm:main,thm:main-symb-bds}.

\subsection{The estimate~\texorpdfstring{\eqref{eqn:main-decay-undiff}}{(7)}}\label{sec:undiff-decay}

Let us recall the decay hypothesis~\eqref{eqn:main-decay-hypo} from~\Cref{thm:main} here: For some $\epsilon > 0$,
\begin{equation}\label{eqn:main-decay-hypo-reprint}
    \sup_{x,t}\jBra{x}^{-(1+\epsilon)}|w(x,t)| \lesssim 1
\end{equation}
where $w(x,t) = u(x+t,t)$ is the expression for the solution in the moving reference frame.  We will show that~\eqref{eqn:main-decay-hypo-reprint} implies the improved decay estimate
\begin{equation}\label{eqn:decay-hypo-improved}
    \sup_{x,t} \jBra{x}^{\min(1+ \frac{3}{2}\epsilon, 2)} |w(x,t)| \lesssim 1
\end{equation}
Thus,~\eqref{eqn:main-decay-hypo-reprint} holds with the improved exponent $\min(1+ \frac{3}{2}\epsilon, 2)$.  After iterating this argument finitely many times (the exact number depending on $\epsilon$), we conclude that $\jBra{x}^{2}|w(x,t)| \lesssim 1$, which is~\eqref{eqn:main-decay-undiff}.  Thus, it suffices to prove that~\eqref{eqn:main-decay-hypo-reprint} implies~\eqref{eqn:decay-hypo-improved}.  To do this, we first observe that it is equivalent to prove the estimate
\begin{equation}\label{eqn:dyadic-space-loc-est}
    \lVert \chi_j^\pm (x)  w(x,t) \rVert_{L^\infty_{x,t}} \lesssim 2^{-\min(1+ \frac{3}{2}\epsilon, 2)j}
\end{equation}
for all $j \geq 0$ (with implicit constant independent of $j$).  Since~\eqref{eqn:BO} is invariant under the change of variables $u(x,t) \mapsto u(-x,-t)$, it is sufficient to prove the estimate for $x > 0$ (i.e. for the $\chi_j^+$ cut-offs).

Let us fix $j \geq 0$, and define 
\begin{equation}\label{eqn:k0-redef}
    k_0 = -\frac{1-\epsilon}{2} j
\end{equation}
as in~\eqref{eqn:k-0-def} in~\Cref{thm:ref-lin-est}.  By writing
\begin{equation*}
    w = w_{\leq k_0} + \sum_{k > k_0} (w^+_k + w^-_k)
\end{equation*}
we see that the decay bound~\eqref{eqn:dyadic-space-loc-est} will follow if we can prove appropriate bounds on the low frequency piece $w_{\leq k_0}$ and the dyadically localized functions $w^+_k$ (the bounds for $w^-_k$ follow by observing that $w$ is real-valued, so $w^-_k(x,t) = \overline{w^+_k(x,t)}$).  

\subsubsection{Spatial decay at low frequencies}\label{sec:sd-low-freq}

We begin by proving spatial localization for $w_{\leq k_0}$ consistent with~\eqref{eqn:dyadic-space-loc-est}.  From~\eqref{eqn:w-def}, we see that $w$ satisfies the equation
\begin{equation}\label{eqn:BO-moving-frame}
    \partial_t w - \partial_x w - H \partial_x^2 w = -\partial_x (w^2)
\end{equation}
It follows that for any $T > 0$
\begin{equation*}
    w(x,t) = e^{(T-t)(\partial_x + H \partial_x^2)} w(x, T) + \int_t^T e^{(s-t)(\partial_x + H \partial_x^2)} \partial_x(w(x, s)^2)\;ds
\end{equation*}
Standard dispersive estimates~\eqref{eqn:BO-disp-basic} together with~\eqref{eqn:main-decay-hypo-reprint} show that the first term vanishes in $L^\infty$ as $T \to \infty$, so we can write
\begin{equation}\label{eqn:w-duhamel-infty}
    w(x,t) = \int_t^\infty e^{(s-t)(\partial_x + H \partial_x^2)} \partial_x(w(x, s)^2)\;ds
\end{equation}
Projecting in space and frequency, we find that
\begin{subequations}\begin{align}
    \chi_j^+(x) w_{\leq k_0}(x,t) =& \int_t^\infty \chi_j^+(x)  e^{(s-t)(\partial_x + H \partial_x^2)} P_{\leq k_0} \partial_x(w^2)\;ds\notag\\
    =& \int_t^\infty \chi_j^+(x)  e^{(s-t)(\partial_x + H \partial_x^2)} P_{\leq k_0} \partial_x \chi^+_{\lesssim j} w^2 \;ds\label{eqn:low-freq-left-waves-term}\\
    &+ \sum_{\ell = j-9}^\infty \int_t^\infty \chi_j^+(x)  e^{(s-t)(\partial_x + H \partial_x^2)} P_{\leq k_0} \partial_x \chi_\ell^+(x) w^2 \;ds \label{eqn:low-freq-right-waves-term}
\end{align}
\end{subequations}
For~\eqref{eqn:low-freq-left-waves-term}, we use inequality~\eqref{eqn:low-freq-left-waves} with $a = 1$ to conclude that
\begin{equation*}\begin{split}
    \lVert \eqref{eqn:low-freq-left-waves-term} \rVert_{L^\infty} \leq& \int_t^\infty \left\lVert \chi_j^+(x)  e^{(s-t)(\partial_x + H \partial_x^2)} P_{\leq k_0} \partial_x \chi^+_{\lesssim j}) \right\rVert_{L^1 \to L^\infty} \lVert w^2(x,s) \rVert_{L^1_x} \;ds\\
    \lesssim& \int_t^\infty \min(2^{-3j}, (s-t)^{-3}) \;ds\\
    \lesssim& 2^{-2j}
\end{split}\end{equation*}
For the second term, we split the integral into the near future $s-t \leq 2^{\ell + 10}$ and far future $s-t \geq 2^{\ell + 10}$ regions.  Using~\Cref{thm:ref-lin-est} in the far future and the standard $L^1 \to L^\infty$ estimate~\eqref{eqn:BO-disp-basic} in the near future, we find that
\begin{equation*}\begin{split}
    \lVert \eqref{eqn:low-freq-right-waves-term} \rVert_{L^\infty} \leq& \sum_{\ell = j-9}^\infty \int_{t+2^{\ell+10}}^\infty \left\lVert \chi_j^+(x)  e^{(s-t)(\partial_x + H \partial_x^2)} P_{\leq k_0} \partial_x \chi_\ell^+(x) w^2(x,s) \right\rVert_{L^1_x} \;ds\\
    &+ \sum_{\ell = j-9}^\infty \int_t^{t+2^{\ell+10}} \lVert e^{(s-t)(\partial_x + H \partial_x^2)}  P_{\leq k_0} \partial_x  \chi_{\ell}^+(x) w^2(x,s) \rVert_{L^1_x} \;ds\\
    \lesssim& \sum_{\ell=j-9}^\infty \int_{t+2^{\ell+10}}^\infty(s-t)^{-3} 2^{-(1+2\epsilon) \ell} \;ds\\
    &+ \sum_{\ell=j-9}^\infty \int_{t}^{t+2^{\ell+10}} (s-t)^{-1/2} 2^{k_0} 2^{-(1+2\epsilon) \ell} \;ds\\
    \lesssim& 2^{-(3+2\epsilon) j} + 2^{k_0} 2^{-(1/2 + 2\epsilon)j}\\
    \lesssim& 2^{-(1+ \frac{3}{2}\epsilon)j}
\end{split}\end{equation*}
Both of these bounds are consistent with~\eqref{eqn:dyadic-space-loc-est}, completing the argument for $w_{\leq k_0}$.

\subsubsection{Spatial decay for $w_k^+$: Reduction to estimates on $\tilde{w}_k^+$}

We now prove that the $w_k^+$ terms have spatial decay consistent with~\eqref{eqn:dyadic-space-loc-est}.  To begin, we recall that by~\eqref{eqn:NFGT-def} and~\eqref{eqn:NG-trans-ref},
\begin{equation*}
    v_k^+ = \left(u_k^+ + B_{k,N}^+(u,u)\right) E_N(\Phi_{\ll_N k})
\end{equation*}
satisfies the equation
\begin{equation*}\begin{split}
    (i\partial_t - \partial_x^2) v_k^+ =& \mathcal{B}_{k,N}^{+,\text{rem}}(u,u) \frac{(\Phi_{\ll_N k})^N}{N!} + \tilde{C}_{k,N}^+(u,u,u) \frac{(-i \Phi_{\ll_N k})^N}{N!}\\
    &+ C_{k,N}^+(u,u,u) E_{N-1}(\Phi_{\ll_N k})  + Q_{k,N}^+(u,u,u,u) E_{N-1}(\Phi_{\ll_N k})
\end{split}\end{equation*}
Let us define
\begin{equation}\label{eqn:Psi-def}
    \Psi(x,t) = \Phi(x+t,t)
\end{equation}
and
\begin{equation}\label{eqn:tilde-w-def}
    \tilde{w}_k^+(x,t) = v_k^+(x+t,t) = \left(w_k^+ + B_{k,N}^+(w,w)\right) E_N(\Psi_{\ll_N k})
\end{equation}
Now, since $w \in L^\infty_tL^1_x$, the antiderivative $\Psi \in L^\infty_{t,x}$.  In particular, $\Psi_{\ll_N k}(x,t)$ is uniformly bounded in $t, x, N,$ and $k$.  Since $E_N(x) \to e^{-ix}$ as $N \to \infty$ uniformly on compact sets, it follows that for $N$ sufficiently large,
\begin{equation*}
    |E_N(\Psi_{\ll_N k}(x,t))| \sim 1
\end{equation*}
uniformly in $x,t,$ and $k$.  Combining this with the bound~\eqref{eqn:B-k-bd}, we find that
\begin{equation*}\begin{split}
    \lVert \chi_j^+(x) {w}_k^+ \rVert_{L^\infty} \sim& \lVert \chi_j^+(x) \tilde{w}_k^+ \rVert_{L^\infty} + O\left(\lVert \chi_j^+(x) B_{k,N}^+(w,w) \rVert_{L^\infty}\right)\\
    \sim& \lVert \chi_j^+(x) \tilde{w}_k^+ \rVert_{L^\infty} + O\left(2^{-(2+\epsilon)j} 2^{-k} \jBra{2^{\min(\epsilon, 1/100) k}}\right)
\end{split}\end{equation*}
and summing over $k > k_0 = -\frac{1-\epsilon}{2}j$ shows that
\begin{equation*}
    \sum_{k > k_0} \lVert \chi_j^+(x) {w}_k^+ \rVert_{L^\infty} \sim \sum_{k > k_0} \lVert \chi_j^+(x) \tilde{w}_k^+ \rVert_{L^\infty} + O\left(2^{(-\frac{3}{2}-\frac{\epsilon}{2})j}\right)
\end{equation*}
since the error term coming from $B^+_{k,N}(w,w)$ is lower order than what we require for~\eqref{eqn:decay-hypo-improved}, the desired result will follow once we show that
\begin{equation}\label{eqn:tildew-req-bds}
    \sum_{k > k_0} \lVert \chi_j^+(x) \tilde{w}_k^+ \rVert_{L^\infty} \lesssim 2^{-(1+\frac{3}{2}\epsilon)j}
\end{equation}

\subsubsection{Spatial decay estimates for $\tilde{w}_k^+$}

We will now prove estimates for $\tilde{w}_k^+$ that are compatible with~\eqref{eqn:tildew-req-bds}.  To do this, we will use the Duhamel representation (which follows from~\eqref{eqn:tilde-w-def} and~\eqref{eqn:NG-trans-ref})
\begin{subequations}\label{eqn:tilde-w-duhamel-exp}\begin{align}
    \chi_j^+(x) \tilde{w}_k^+(t) =& i\chi_j^+(x) \int_t^\infty e^{(s-t)(\partial_x - i\partial_x^2)} \mathcal{B}^{+,\text{rem}}_{k,N} \frac{(-i\Psi_{\ll_N k})^N}{N!}\;ds\label{eqn:tw-calBrem-term}\\
    &+i\chi_j^+(x) \int_t^\infty e^{(s-t)(\partial_x - i\partial_x^2)} C^{+}_{k,N} E_{N-1}(\Psi_{\ll_N k}) \;ds\label{eqn:tw-C-term}\\
    &+i\chi_j^+(x) \int_t^\infty e^{(s-t)(\partial_x - i\partial_x^2)} \tilde{C}^{+}_{k,N} \frac{(-i\Psi_{\ll_N k})^N}{N!}\;ds\label{eqn:tw-tC-term}\\
    &+i\chi_j^+(x) \int_t^\infty e^{(s-t)(\partial_x - i\partial_x^2)} Q^{+}_{k,N}  E_{N-1}(\Psi_{\ll_N k}) \;ds\label{eqn:tw-Q-term}
\end{align}\end{subequations}
Thus, it suffices to estimate each these integrals.  Before we begin, we observe that because of the $N$ dependent frequency localization of $\Psi$, each of the integrals is localized to positive frequencies $\xi \sim 2^k$, so we may apply the results of~\Cref{cor:Schro-prop}.

For~\eqref{eqn:tw-calBrem-term}, we note that by~\eqref{eqn:Phi-BO},
\begin{equation*}
    \lVert \Psi_{\ll_N k} \rVert_{L^\infty} \lesssim 1, \qquad \qquad \lVert \chi_j^+(x) \Psi_{\ll_N k} \rVert_{L^\infty} \lesssim 2^{\epsilon j}
\end{equation*}
so the $(\Psi_{\ll_N k})^N$ term will provide enough extra decay (for $N$ large) to compensate for the slower decay of the quadratic remainder term $B_{k,N}^{+,\text{rem}}(w,w)$.  To make this rigorous, we first introduce the dyadic (spatial) decomposition, and then apply the results of~\Cref{cor:Schro-prop} to obtain
\begin{equation*}\begin{split}
    \lVert \eqref{eqn:tw-calBrem-term} \rVert_{L^\infty} \lesssim_N& \left\lVert \chi_j^+(x) \int_t^\infty e^{(s-t)(\partial_x - i\partial_x^2)} P^+_{[k\pm20]} \chi_{\lesssim j}^+(x) \mathcal{B}^{+,\text{rem}}_{k,N} \Psi_{\ll_N k}^N\;ds \right\rVert_{L^\infty}\\
    &+ \sum_{\ell \geq j-9} \left\lVert \chi_j^+(x) \int_t^\infty e^{(s-t)(\partial_x - i\partial_x^2)} P^+_{[k\pm20]} \chi^+_{\ell}(x) \mathcal{B}^{+,\text{rem}}_{k,N} \Psi_{\ll_N k}^N\;ds\right\rVert_{L^\infty}\\
    \lesssim_{N,M}& \int_t^\infty \min\left(2^{-M/2j}, (s-t)^{-M/2} \jBra{2^k}^{-2}\right) \lVert \mathcal{B}_{k,N}^{+,\text{rem}} \rVert_{L^1} \lVert \Psi \rVert_{L^\infty}^N\;ds\\
    &+ \sum_{\ell \geq j - 9} \int_{t+2^{\ell+10}\jBra{2^k}^{-1}}^\infty (s-t)^{-M/2}\jBra{2^k}^{-{M/2}} \lVert \mathcal{B}_{k,N}^{+,\text{rem}} \rVert_{L^1} \lVert \Psi \rVert_{L^\infty}^N\;ds\\
    &+ \sum_{\ell \geq j - 9} \int_t^{t+2^{\ell+10}\jBra{2^k}^{-1}}\!\!\!\!\!\!\!\!\!\! (s-t)^{-1/2} 2^{\ell/2} \lVert \chi^+_{\sim \ell}\mathcal{B}_{k,N}^{+,\text{rem}} \rVert_{L^2} \lVert \chi^+_{\sim \ell} \Psi_{\ll_N k} \rVert_{L^\infty}^N\;ds
\end{split}\end{equation*}
The leading order contribution comes from the last term.  Using~\eqref{eqn:calB-rem-bd}, we estimate
\begin{equation}
    \lVert \eqref{eqn:tw-calBrem-term} \rVert_{L^\infty} \lesssim \sum_{\ell > j-9} 2^{-(1/2+(N+1)\epsilon)\ell} c_k \lesssim 2^{-(1/2+(N+1)\epsilon)j} c_k \label{eqn:tw-calBrem-bd}
\end{equation}
which is compatible with~\eqref{eqn:tildew-req-bds} for $N$ sufficiently large (depending on $\epsilon$).  Turning to~\eqref{eqn:tw-C-term}, we use~\eqref{eqn:C-k-bd} to find that
\begin{equation*}\begin{split}
    \lVert \eqref{eqn:tw-C-term} \rVert_{L^\infty} \!\!\lesssim&\!\! \sum_{\ell \geq j - 9} \left\lVert \int_t^{t + 2^{\ell+10}\jBra{2^k}^{-1}} \!\!\!\!\!\!\!\!\!\!\!e^{(s-t)(\partial_x - i \partial_x^2)} P_{[k\pm20]} \chi_\ell^+(x) C_{k,N}^+ E_{N-1}(\Psi_{\ll_N k})\;ds\right\rVert_{L^\infty}\\
    &+ \{\text{better}\}\\
    \lesssim& 2^{-(3/2 + \epsilon)j} \left( 2^{(2\min(1/100, \epsilon) -1/2)k}, 2^{-k}\right)
\end{split}\end{equation*}
Since $\tilde{C}_{k,N}^+$ obeys the same bounds as $C_{k,N}^+$, we also get the same bound for~\eqref{eqn:tw-tC-term}. Summing over $k > k_0$ gives
\begin{equation}\label{eqn:tw-C-bd}
    \sum_{k > k_0} \lVert \eqref{eqn:tw-C-term} \rVert_{L^\infty} + \lVert \eqref{eqn:tw-tC-term} \rVert_{L^\infty} \lesssim 2^{-(1 + 3/2\epsilon)j}
\end{equation}
as required by~\eqref{eqn:tildew-req-bds}.  Finally, for the quartic terms~\eqref{eqn:tw-Q-term}, we find that
\begin{equation*}\begin{split}
    \lVert \eqref{eqn:tw-Q-term} \rVert_{L^\infty} \!\!\lesssim&\!\! \sum_{\ell \geq j - 9} \left\lVert \int_t^{t + 2^{\ell+10}\jBra{2^k}^{-1}} \!\!\!\!\!\!\!\!\!\!\! e^{(s-t)(\partial_x - i \partial_x^2)} P_{[k\pm 20]} \chi_\ell^+(x) Q_{k,N}^+ E_{N-1}(\Psi_{\ll_N k})\;ds\right\rVert_{L^\infty}\\
    &+ \{\text{better}\}\\
    \lesssim& 2^{-(5/2 + \epsilon)j} 2^{-k}\jBra{2^{(\min(1/100, \epsilon) - 1/2)k}}
\end{split}\end{equation*}
so
\begin{equation}\label{eqn:tw-Q-bd}
    \sum_{k > k_0} \lVert \eqref{eqn:tw-Q-term} \rVert_{L^\infty}  2^{-(2+ 3/2\epsilon)j}
\end{equation}
Combining~\eqref{eqn:tw-calBrem-bd},~\eqref{eqn:tw-C-bd} and~\eqref{eqn:tw-Q-bd} gives~\eqref{eqn:tildew-req-bds} and completes the proof of~\Cref{thm:main}.

\subsection{Modifications to prove~\texorpdfstring{\Cref{thm:main-symb-bds}}{Theorem 3}}\label{sec:symb-bds}
The same argument given above can be modified to prove~\Cref{thm:main-symb-bds}.  The basic bootstrapping strategy remains unchanged, but some of the terms in the argument must be modified.  We will briefly sketch the modifications.

\subsubsection{Spatial decay at low frequency}

By differentiating~\eqref{eqn:BO-moving-frame}, we see that
\begin{equation*}
    (\partial_t - \partial_x - H\partial_x^2) \partial_x^m w = \partial_x^{m+1} (w^2)
\end{equation*}
Arguing as in~\Cref{sec:sd-low-freq} and using $a = m+1$ in~\eqref{eqn:low-freq-left-waves}, we see that
\begin{equation}
    \lVert \chi_j^+(x) \partial_x^m P_{\leq k_0} w \rVert_{L^\infty} \lesssim \min(2^{-(1+3/2\epsilon)j}, 2^{-(m+2)j})
\end{equation}

\subsubsection{Reduction via normal forms}

Differentiating~\eqref{eqn:NFGT-def}, we see that
\begin{equation*}\begin{split}
    \partial_x v_k^+    =& (\partial_x u_k^+ + \partial_x B_{k,N}^+(u,u)) E_{N}(\Phi_{\ll_N k})\\
                         &- i u_{\ll_N k} (u_k^+ + B_{k,N}^+(u,u)) E_{N-1}(-\Phi_{\ll_N k})
\end{split}\end{equation*}
Recalling the Leibnitz rule for pseudoproducts and noting that $u_k^+ = \partial_x^{-1} \partial_x u_k^+$, we see that this is of the form
\begin{equation*}
    \partial_x v_k^+    = \partial_x u_k^+ E_{N}(\Phi_{\ll_N k}) + \{\text{pseudoproducts multiplied by } E_M(\Phi_{\ll_N k})\}
\end{equation*}
Higher derivatives of $v_k^+$ also have this form.  Thus, passing to a moving reference frame, we find that
\begin{equation*}
    \lVert \chi_j^+(x) \partial_x^m w_k^+ \rVert_{L^\infty} = \lVert \chi_j^+(x) \partial_x^m \tilde{w}_k^+ \rVert_{L^\infty} + \{\text{lower order}\}
\end{equation*}
so the required bound~\eqref{eqn:symb-decay-bd} will follow if we can prove appropriate bounds on $\partial_x^m \tilde{w}_k^+$.

\subsubsection{Spatial decay at high frequency}

Differentiating~\eqref{eqn:NG-trans-ref}, we find that
\begin{equation*}\begin{split}
	(i\partial_t- i\partial_x - \partial_x^2) \partial_x^m w_k^+ =& i^N\sum_{a=0}^m C_{a,m} \partial_x^a \mathcal{B}_{k,N}^{+,\text{rem}}(u,u) \partial_x^{m-a}\frac{\Phi_{\ll_N k}}{N!} \\
    &+ \{\text{better}\}\\
    =& \partial_x^m \mathcal{B}_{k,N}^{+,\text{rem}}(u,u)\frac{\Phi_{\ll_N k}}{N!} + \{\text{better}\}
\end{split}\end{equation*}
where $\{\text{better}\}$ denotes pseudoproducts of differential order $0$ which are cubic or higher order in $u, \partial_x u, \cdots, \partial_x^m u$ -- these terms can be estimated using straightforward modifications of the previous arguments.  For the remaining terms (which are bilinear in $u$), we distribute the derivative to find that
\begin{equation*}
    \partial_x^m \mathcal{B}_{k,N}^{+,\text{rem}}(u,u) =  (H+i)\partial_x^{m+1} u_{\ll_N k} u_k^+ + 2iu_{\ll_N k} \partial_x^{m+1} u_k^+ + \{\text{better}\}
\end{equation*}
where the $\{\text{better}\}$ terms contain only derivatives of order $m$ or lower.  The main terms obey the same bounds given in~\eqref{eqn:calB-rem-bd-L1} (and hence can be handled using the decay of $\Psi_{\ll_N k}^N$ as in~\eqref{eqn:tw-calBrem-bd}), while the remaining terms satisfy even better bounds using~\eqref{eqn:symb-decay-hypo}.

\bibliography{sources}
\bibliographystyle{plain}

\end{document}